\documentclass{amsart}

\usepackage{amssymb}
\usepackage{amsmath}
\usepackage[unicode,linkcolor=blue,hyperindex]{hyperref}
\usepackage{url,breakurl}
\usepackage{multirow}
\usepackage{multicol}
\usepackage{graphicx}
\usepackage{enumerate}

\usepackage{geometry}

\numberwithin{equation}{section} \numberwithin{figure}{section}

{\theoremstyle{remark} \newtheorem{remark}{Remark}}
\newtheorem{definition}{Definition}[section]
\newtheorem{proposition}{Proposition}[section]

\newtheorem{theorem}{Theorem}[section]
\newtheorem{corollary}{Corollary}[section]
\newtheorem{lemma}{Lemma}[section]

\newcommand{\Wzop}{ \mbox{ \raisebox{7.2pt} {\tiny$\circ$} \kern-10.7pt} {W}^1_p }
\newcommand{\Wzoinfty}{\mbox{ \raisebox{7.2pt} {\tiny$\circ$} \kern-10.7pt} {W}^1_\infty }

\def\calF{\mathcal{F}}

\def\calO{\mathcal{O}}
\def\calC{\mathcal{C}}

\newcommand{\polR}{\mathbb{R}}
\newcommand{\Rd}{{\mathbb{R}^d}}
\newcommand{\vn}{\textbf{n}}
\newcommand{\tHs}{\widetilde H^s(\Omega)}

\newcommand{\supp}{\textrm{supp~}}

\usepackage{color}
\newcommand{\red}[1]{{\color{black}#1}}
\newcommand{\jp}[1]{{\color{black}#1}}
\newcommand{\rhn}[1]{{\color{black}#1}}

\begin{document}
\title[Besov regularity for the fractional Laplacian in Lipschitz domains]{Besov regularity for the Dirichlet integral fractional Laplacian in Lipschitz domains}

\author[J.P.~Borthagaray]{Juan Pablo~Borthagaray}
\address[J.P.~Borthagaray]{Departamento de Matem\'atica y Estad\'istica del Litoral, Universidad de la Rep\'ublica, Salto, Uruguay. Current address: Centro de Matem\'atica, Universidad de la Rep\'ublica, Montevideo, Uruguay.}
\email{jpb@cmat.edu.uy}
\thanks{JPB has been supported in part by NSF grant DMS-1411808 and Fondo Vaz Ferreira grant 2019-068.}

\author[R.H.~Nochetto]{Ricardo H.~Nochetto}
\address[R.H.~Nochetto]{Department of Mathematics and Institute for
Physical Science and Technology, University of Maryland, College
Park, MD 20742, USA}
\email{rhn@math.umd.edu}
\thanks{RHN has been supported in part by NSF grants DMS-1411808 and DMS-1908267.}
 
\begin{abstract}
We prove Besov regularity estimates for the solution of the Dirichlet problem involving the integral fractional Laplacian of order $s$ in bounded Lipschitz domains $\Omega$:
\red{
\[ \begin{aligned}
\|u\|_{\dot{B}^{s+r}_{2,\infty}(\Omega)} \le C \|f\|_{L^2(\Omega)}, &
    \quad r = \min\{s,1/2\}, & \quad \mbox{if } s \neq 1/2, \\
\|u\|_{\dot{B}^{1-\epsilon}_{2,\infty}(\Omega)} \le C \|f\|_{L^2(\Omega)}, &
    \quad \jp{\epsilon \in (0,1),} & \quad \mbox{if } s = 1/2,
\end{aligned} \]
\jp{with explicit dependence of $C$ on $s$ and $\epsilon$.} These estimates are} consistent with the regularity on smooth domains and show that there is no loss of regularity due to Lipschitz boundaries. The proof uses elementary ingredients, such as the variational structure of the problem and the difference quotient technique.
\end{abstract}
 
\maketitle

\section{Introduction}

Given $s \in (0,1)$, we consider the integral fractional Laplacian of order $s$,
\begin{equation}
\label{eq:defofLaps}
  (-\Delta)^s v(x) = C(d,s) \mbox{ p.v.} \int_\Rd \frac{v(x)-v(y)}{|x-y|^{d+2s}} d y, \qquad C(d,s) = \frac{2^{2s} s \Gamma(s+\frac{d}{2})}{\pi^{d/2}\Gamma(1-s)}.
\end{equation}
In this work, we study the regularity of the solution to the homogeneous Dirichlet problem
\begin{equation} \label{eq:Dirichlet}
\left\lbrace \begin{array}{rl}
(-\Delta)^s u  = f & \mbox{in } \Omega, \\
u = 0 & \mbox{in } \Omega^c = \Rd \setminus \overline\Omega ,
\end{array} \right.
\end{equation}
where $\Omega \subset \Rd$ is a bounded Lipschitz domain. 
Our method follows ideas from Savar\'e \cite{Savare98}, who adapted the classical difference quotient technique of Nirenberg \cite{Nirenberg} to develop a clever, $L^2$-based variational argument to deal with regularity of integer-order problems on Lipschitz domains. This led to an elementary approach to the regularity theory developed by Jerison and Kenig \cite{JerisonKenig:1995} for the Laplace equation in Lipschitz domains, and extends to other linear and nonlinear elliptic PDEs. Our technique differs from \cite{Savare98} in two fundamental aspects: it uses second-order differences in the characterization of Besov spaces, and a bootstrap argument to obtain optimal regularity estimates.

Regularity of solutions of the homogeneous fractional Dirichlet problem \eqref{eq:Dirichlet} on bounded domains has been analyzed, for example, in \cite{AbatangeloRosOton,AbelsGrubb,Grubb15,ROSe14,VishikEskin}. We briefly comment on the results in those references. Techniques based on Fourier analysis, such as the ones employed by Vi\v{s}ik and \`Eskin \cite{Eskin, VishikEskin} or Grubb \cite{Grubb15}, allow for a full characterization of mapping properties of the integral fractional Laplacian of functions supported in $\Omega$. However, such arguments typically require the domain $\Omega$ to be smooth; the recent work by Abels and Grubb \cite{AbelsGrubb} introduces a method to handle nonsmooth coordinate changes that leads to regularity results for domains with $C^{1+\beta}$ boundary with $\beta > 2s$. References \cite{AbatangeloRosOton,ROSe14} deal with integral operators with translation-invariant kernels. Ros-Oton and Serra \cite{ROSe14}, by developing an analog of the Krylov boundary Harnack method for \eqref{eq:Dirichlet}, derived H\"older regularity estimates on bounded Lipschitz domains satisfying an exterior ball condition; these, in turn, can be reinterpreted as (weighted) Sobolev estimates for $u$ in terms of H\"older norms of $f$ \cite{AcBo17}. 
Abatangelo and Ros-Oton \cite{AbatangeloRosOton} improved upon the results from \cite{ROSe14}
in the case the domain $\Omega$ is of class $C^{1+\beta}$ with $\beta > s$. 
Finally, let us also point out interior regularity estimates in \cite{BiWaZu17,Cozzi17, Faustmann:20}.

A major difference between problem \eqref{eq:Dirichlet} and its local second-order counterpart is the lack of explicit solutions. Nevertheless, if the domain under consideration is a ball, there is a vast number of examples based on expansions with respect to Meijer $G$-functions, cf. \cite{Dyda2016} (see also \cite{ABBM18} for related results in one-dimensional domains). A striking example with right hand side $f=1$ corresponds to a $2s$-stable L\'evy process in $\Omega$, in which case the solution $u$ is the first exit time. Concretely, if $\Omega = D_r(0)$ is a ball of radius $r$ centered at the origin, then it holds that \cite{Getoor61}
\begin{equation} \label{eq:example}
  u (x) = \frac{2^{-2s} \Gamma\left(\frac{d}{2}\right)}{\Gamma\left(\frac{d+2s}{2}\right) \Gamma(1+s)} \left(r^2 - |x|^2 \right)^s_+ \quad \Rightarrow\quad
  (-\Delta)^s u \equiv 1 \mbox{ in } \Omega.
\end{equation}
This explicit solution is important for a number of reasons. First, it serves as a guide for boundary regularity. In fact, both the domain and the right hand side are smooth, yet the solution $u$ satisfies
\begin{equation}\label{eq:lack-regularity}
u \in \bigcap_{\epsilon > 0} \widetilde H^{s+1/2-\epsilon}(\Omega), \quad \rhn{u \notin \widetilde{H}^{s+1/2}(\Omega);}
\end{equation}
moreover $u\in\dot{B}^{s+1/2}_{2,\infty}(\Omega)$.
We refer to Section \ref{sec:preliminaries} for definitions of spaces and fractional-order norms.
Second, the function $u$ in \eqref{eq:example} exhibits the boundary behavior
\[
u (x) \simeq \mbox{dist}(x, \partial \Omega)^s ,
\]
that is typical of solutions to fractional-order elliptic problems. This algebraic singularity arises regardless of the smoothness of the domain. In particular, it seems plausible that such a weak singularity distributed along $\partial \Omega$ have a stronger effect on the integrability of the difference quotients on the domain than the presence of a reentrant corner could have. This observation is supported by the recent work \cite{Gimperlein:19}, where problem \eqref{eq:Dirichlet} is studied in polygons and asymptotic expansions of its solutions near edges and vertices are given. 

To the best of our knowledge, there are no regularity estimates for \eqref{eq:Dirichlet} up to the boundary valid for arbitrary Lipschitz domains.
\red{We point out, however, to references providing representation formulas for $s$-harmonic functions \cite{Bogdan:99}, and estimates on Green functions in \cite{Jakubowski:02} --following ideas from \cite{Bogdan:00}--.}
Numerical evidence indicates that the high regularity of $\partial\Omega$ assumed in \cite{Grubb15} can be drastically weakened. In \cite{BdPM}, through the study of eigenvalue problems for the integral fractional Laplacian using the finite element method on uniform meshes, similar experimental orders of convergence were obtained on an $L$-shaped domain and in smooth domains. Additionally, in \cite{BoCi18} some experiments on a family of domains with reentrant corners with angle $\theta \in (\pi, 2\pi)$ are carried out. The results in that paper indicate that the orders of convergence in $H^1$-norm do not deteriorate as $\theta \to 2\pi$. These phenomena are in striking contrast with the classical (local) Dirichlet problem, in which reentrant corners affect the regularity of solutions. \red{For problems posed on cones, there is a high sensitivity to the opening solid angle in the behavior of $s$-harmonic functions as $s \to 1$ \cite{Terracini:18}.
}

An important missing information about the regularity of \eqref{eq:Dirichlet} is a shift theorem that accounts for the pick-up of $2s$ derivatives associated with operators of order $2s$ such as the fractional Laplacian $(-\Delta)^s$. Unfortunately, the estimate
\begin{equation}\label{eq:false}
\|u\|_{H^{2s}(\Omega)} \le C \|f\|_{L^2(\Omega)}
\end{equation}
with $C=C(\Omega)$ is false on bounded domains if $s \in [1/2, 1)$ according to \eqref{eq:lack-regularity}; we refer to Section \ref{sec:preliminaries} for definitions and some characterizations of Sobolev and Besov spaces. Estimates of this type are critical in the analysis of discretization schemes, such as the finite element method, in domains $\Omega$ without smooth boundary. Moreover, duality arguments such as the Aubin-Nitsche trick rely on shift estimates; they yield convergence rates in norms weaker than the energy norm. In this paper we derive shift theorems with $L^2(\Omega)$ or weaker regularity of the forcing function $f$ and valid on Lipschitz domains $\Omega$ in $\Rd$. Therefore, our results are relevant in applications on polytopal domains.

The unique weak solution $u$ to \eqref{eq:Dirichlet} is the minimum of the functional $\calF : \tHs \to \polR$,
\begin{equation} \label{eq:def-functional}
\calF (v) = \frac12 |v|_{H^s(\Rd)}^2 - \rhn{\langle f, v \rangle,}
\end{equation}
\rhn{where $\langle f, v \rangle$ denotes the duality pairing for $f\in H^{-s}(\Omega)$.}
For clarity, we shall adopt the following notation for the quadratic and linear components of $\calF$
\begin{equation} \label{eq:def-subfunctionals}
\calF_2 (v) := \frac12 |v|_{H^s(\Rd)}^2, \qquad \rhn{\calF_1 (v) := \langle f, v \rangle.}
\end{equation}
It is clear that if $f \in H^{-s}(\Omega)$, then the solution to \eqref{eq:Dirichlet} verifies $u \in \tHs$ and
\begin{equation} \label{eq:well-posedness}
|u|_{H^s(\Rd)} \le \|f\|_{H^{-s}(\Omega)}.
\end{equation} 

The main goal of this manuscript is to prove the following two shift theorems.

\red{
\begin{theorem}[Besov regularity for $L^2$-data] \label{thm:regularity-solutions}
Let $\Omega$ be a bounded Lipschitz domain and $f \in L^2(\Omega)$. 
If $s \neq 1/2$, then the solution $u$ to \eqref{eq:Dirichlet} satisfies $u \in \dot{B}^{s+r}_{2,\infty}(\Omega)$ with $r = \min \{ s , 1/2\}$, and 
\begin{equation}\label{eq:regularity-u-low}
\| u\|_{\dot{B}^{s+r}_{2,\infty}(\Omega)} \le \frac{C(\Omega,d)}{\sqrt{|1-2s|}} \| f \|_{L^2(\Omega)} .
\end{equation}
On the other hand, the solution for $s = 1/2$ satisfies $u \in \dot{B}^{1-\epsilon}_{2,\infty}(\Omega)$ \jp{for every $0<\epsilon<1$ and}
\begin{equation} \label{eq:regularity-u-s1/2}
\| u \|_{\dot{B}^{1-\epsilon}_{2,\infty}(\Omega)} \le \frac{C(\Omega,d)}{\sqrt{\epsilon}} \| f \|_{L^2(\Omega)}.
\end{equation}
\end{theorem}
}
\rhn{
It is worth noticing that \eqref{eq:regularity-u-s1/2} for $s=1/2$ is consistent with \eqref{eq:regularity-u-low} for $s=r=(1-\epsilon)/2$. Moreover, \eqref{eq:regularity-u-low} for $s=r\in(0,1/2)$ is of the form \eqref{eq:false} but with $H^{2s}(\Omega)$ replaced by $\dot{B}^{2s}_{2,\infty}(\Omega) \subsetneq \widetilde{H}^{2s-\epsilon}(\Omega)$ for any $0<\epsilon <2s$, whereas \eqref{eq:regularity-u-low} for $s\in (1/2,1)$ matches \eqref{eq:lack-regularity}.}
Our technique exploits the variational structure of \eqref{eq:Dirichlet} and uses the difference quotient technique of Nirenberg \cite{Nirenberg}; it is thus conceptually elementary. It hinges on an approach introduced by Savar\'e \cite{Savare98} for the classical Laplace operator in Lipschitz domains, but it has two important differences. First, we need to deal with Besov spaces with differentiability order $\sigma\in (0,2)$, instead of $\sigma\in (0,1)$, and corresponding second-order difference quotients in cones. Second, the full regularity pick-up in \eqref{eq:regularity-u-low} and \eqref{eq:regularity-u-s1/2} requires a boostrap argument. \rhn{Our technique does not extend to data $f$ more regular than $L^2$, but it does lead Theorem \ref{prop:regularity-solutions-rough} for more singular data. We refer to the recent papers \cite{BoLiNo:Barrett-22,BoLiNo22} which prove the following {\it optimal shift property} with a novel technique that extends to quasi-linear problems:
\begin{equation}\label{eq:optimal-shift}
  \| u\|_{\dot{B}^{s+1/2}_{2,\infty}(\Omega)} \le C(\Omega,d) \| f \|_{B^{-s+1/2}_{2,1}(\Omega)}
  \quad \forall\, s \in (0,1).
\end{equation}
}

\begin{theorem}[Besov regularity for rough data] \label{prop:regularity-solutions-rough}
Let $\Omega$ be a bounded Lipschitz domain. If $s \in (1/2,1)$ and $f \in B^{-s+1/2}_{2,1}(\Omega)$, then the solution $u$ to \eqref{eq:Dirichlet}
satisfies $u \in \dot{B}^{s+1/2}_{2,\infty}(\Omega)$ with
\begin{equation}\label{eq:s+1/2}
\|u\|_{\dot{B}^{s+1/2}_{2,\infty}(\Omega)} \rhn{\le C(\Omega,d)} \|f\|_{B^{-s+1/2}_{2,1}(\Omega)}.
\end{equation}
\end{theorem}

\red{The} following result then follows by an interpolation argument.
\red{
\begin{corollary}[intermediate Besov regularity] \label{cor:intermediate}
If $s \in (0,1) \setminus \{ 1/2 \}$ and  $f \in B^{-s+\theta}_{2,q}(\Omega)$ for
$0 < \theta < \min \{s , 1/2 \}$ and $q \in [1,\infty]$, then there holds
\begin{equation}\label{eq:s+theta}
\| u\|_{\dot{B}^{s+\theta}_{2,q}(\Omega)} \jp{\le C(\Omega,d,q,s)} \| f \|_{B^{-s+\theta}_{2,q}(\Omega)},
\end{equation}
\jp{where $C(\Omega,d,q,s) = C (1-2s)^{\theta/2s}$ for $s\in(0,1/2)$ and $C(\Omega,d,q,s) = C$ for $s\in (1/2,1)$ and $C$ depends on $(\Omega,d,q)$.}
If $s = 1/2$ then, for all $\theta, \epsilon \in (0,1/2)$ and $q \in [1,\infty]$, there holds
\begin{equation}\label{eq:1/2+theta}
\| u\|_{\dot{B}^{1/2+\theta (1- 2\epsilon)}_{2,q}(\Omega)} \jp{\le \frac{C(\Omega,d,q)}{\epsilon^\theta}} \| f \|_{B^{-1/2+\theta}_{2,q}(\Omega)}.
\end{equation}
\end{corollary}
}

We first observe the distinct role of the third index $q$ in \eqref{eq:s+1/2} and \red{\eqref{eq:s+theta}/\eqref{eq:1/2+theta}.}
The proof of \eqref{eq:s+1/2} is constructive, in fact a modification of that of Theorem \ref{thm:regularity-solutions},
and yields the conjugate values $q=\infty$ and $q=1$ that arise by duality.
In contrast, \eqref{eq:s+theta} \red{and \eqref{eq:1/2+theta} contain} the same index $q\in[1,\infty]$ on both sides of the \red{estimates
because they are} a consequence of operator interpolation theory; \jp{thus $C$ depends on $q$.} In particular, setting $q=2$ in \eqref{eq:s+theta}
yields the Sobolev regularity
\[
f \in H^{-s+\theta}(\Omega), \ 0 < \theta < \min \{s , 1/2 \}  \quad \Rightarrow \quad u \in \widetilde{H}^{s+\theta}(\Omega), \ | u|_{\widetilde{H}^{s+\theta}(\Omega)} \le C \| f \|_{H^{-s+\theta}(\Omega)}.
\]
We point out that the constant above blows up as $\theta \to \min \{s , 1/2 \}$. This is consistent with either \eqref{eq:regularity-u-low} (for \red{$s \in (0, 1/2)$}) or \eqref{eq:s+1/2} (for $s \in (1/2, 1)$): writing $\|f\|_* = \| f \|_{L^2(\Omega)}$ if $s \in (0,1/2]$ and 
$\|f\|_* = \|f\|_{B^{-s+1/2}_{2,1}(\Omega)}$ if $s \in (1/2,1)$, and
using the continuity of the embedding $\dot{B}^{s+\min \{s , 1/2 \}}_{2,\infty}(\Omega) \subset \widetilde{H}^{s + \min \{s , 1/2 \} -\epsilon}(\Omega)$ for all $\epsilon > 0$ \red{(cf. Lemma \ref{lemma:embedding-Sobolev-Besov} (embeddings between Besov and Sobolev spaces) below)}, we improve upon \cite[Propositions 3.6 and 3.11]{AcBo17}
\red{
\[ \begin{split}
\|u\|_{\widetilde{H}^{s + \min \{s , 1/2 \} -\epsilon}(\Omega)} \lesssim \frac1{\sqrt \epsilon} \|u\|_{\dot{B}^{s+\min \{s , 1/2 \}}_{2,\infty}(\Omega)} \lesssim \frac1{\sqrt \epsilon} \| f \|_* \quad \mbox{ if } s \neq 1/2, \\
\|u\|_{\widetilde{H}^{1 - 2\epsilon}(\Omega)} \lesssim \frac1{\sqrt \epsilon} \|u\|_{\dot{B}^{1 - \epsilon}_{2,\infty}(\Omega)}  \lesssim \frac1{\epsilon} \| f \|_*  \quad \mbox{ if } s = 1/2.
\end{split} \]
}

A comparison between Theorems \ref{thm:regularity-solutions} and \ref{prop:regularity-solutions-rough}
is in order as well. The gain of $2s$ derivatives in \eqref{eq:regularity-u-low} for $f\in L^2(\Omega)$
and \red{$s\in(0, 1/2)$} extends to $f$ in the negative Besov space $B^{-s+\theta}_{2,q}(\Omega)$ and any \red{$s\neq 1/2$}
in \eqref{eq:s+theta}, provided
$\theta < \min\{s,1/2\}$. We regard \eqref{eq:s+theta} as an adequate replacement for the shift estimate
\eqref{eq:false}. On the other hand, if $s \in (1/2,1)$ then the differentiability limit $s+1/2$ of
\eqref{eq:regularity-u-low} can be achieved with rough data $f\in B^{-s+1/2}_{2,1}(\Omega)$ instead of
$f\in L^2(\Omega)$. \red{The case $s=1/2$ is somewhat special: from our technique, we can prove the suboptimal estimate \eqref{eq:regularity-u-s1/2}. In subsequent work \rhn{\cite{BoLiNo:Barrett-22,BoLiNo22}, we improve upon \eqref{eq:regularity-u-s1/2} and show \eqref{eq:optimal-shift}, which is actually valid for $s=1/2$.}}

Regularity estimates for Dirichlet problems with non-zero exterior data can be derived immediately upon combining the regularity results for the homogeneous problem with mapping properties of the integral fractional Laplacian. As an illustration of such results, we have the following.

\begin{corollary}[Besov regularity for non-homogeneous problem] \label{cor:non-homogeneous}
Let $\Omega$ be a bounded Lipschitz domain \rhn{and data $(f,g)$ satisfy the following assumptions:}
\begin{itemize}
\item If $s \in (0,1/2]$, let $f \in L^2(\Omega)$ and $g \in H^{2s}(\Rd)$.
\item If $s \in (1/2,1)$, let $f \in B^{-s+1/2}_{2,1}(\Omega)$ and $g \in H^{s+1/2}(\Rd)$.
\end{itemize}
Then, the unique weak solution $u$ of the non-homogeneous problem
\begin{equation*} \label{eq:non-homogeneous}
\left\lbrace \begin{array}{rl}
(-\Delta)^s u  = f & \mbox{in } \Omega, \\
u = g & \mbox{in } \Omega^c ,
\end{array} \right.
\end{equation*}
satisfies $u \in B^{s+\min\{s,1/2\}}_{2,\infty}(\Rd)$ \red{if $s \neq 1/2$, while $u \in B^{1-\epsilon}_{2,\infty}(\Rd)$ for every $\epsilon \in (0,1/2)$ if $s=1/2$.}
\end{corollary}

The simplicity of Corollary \ref{cor:non-homogeneous} comes at the expense of its sharpness. In fact, we only need that either $(-\Delta)^s g \in L^2(\Omega)$ for $s\in (0,1/2]$ or $(-\Delta)^s g \in B^{-s+1/2}_{2,1}(\Omega)$ for $s \in (1/2,1)$. This property does not require much regularity of $g$ in $\Omega^c$ according to \eqref{eq:defofLaps}.

As shown by Savar\'e \cite{Savare98} for the classical $p$-Laplacian, the regularity technique of this paper also applies to quasilinear fractional operators such as the $(p,s)$-Laplacian ($s \in (0,1)$, $p \in (1,\infty)$),
\[
(-\Delta)^s_p v(x) := \mbox{ p.v.} \int_\Rd \frac{|v(x)-v(y)|^{p-2} (v(x) - v(y)) }{|x-y|^{d+ps}} d y,
\]
for which we are not aware of any Sobolev regularity estimates up to the boundary of the domain.
In \cite{BoLiNo22}, we derive regularity estimates for the associated Dirichlet problem and provide a priori error estimates for its finite element discretization; \rhn{we also refer to the survey \cite{BoLiNo:Barrett-22}.}

\subsection*{Outline of the paper}
Let us briefly describe the organization of the paper. Section \ref{sec:preliminaries} collects preliminary material regarding function spaces, the crucial notion of regularity of functionals, and includes some additional definitions useful for the remainder of the manuscript. In Section \ref{sec:regularity} we discuss regularity of the functionals $\calF_1$ and $\calF_2$, which are essential for proving our main results. We conclude with the proofs of Theorem \ref{thm:regularity-solutions} in Section \ref{sec:Besov-regularity} and Theorem \ref{prop:regularity-solutions-rough} and Corollary \ref{cor:intermediate} in Section \ref{sec:less-regular-f}.

\section{Preliminaries and definitions} \label{sec:preliminaries}
This section collects some preliminary results we shall need. We define the function spaces that we shall use in the sequel, provide some characterizations by means of translation operators and discuss the relation between these translations and the regularity of solutions of our model problem.

\subsection{Function spaces}
Here we set the notation about fractional-order Sobolev spaces and Besov spaces and list some of their basic properties that we shall use.

\begin{definition}[fractional Sobolev spaces]
Let $\Omega \subset \Rd$ and $\sigma \in(0,1)$ be given. The fractional Sobolev space $H^{\sigma}(\Omega)$ is defined by
\[
H^{\sigma}(\Omega) := \left\{ v \in L^2(\Omega) \colon |v|_{H^\sigma(\Omega)} < \infty \right\},
\]
where $|\cdot|_{H^\sigma(\Omega)}$ is the Aronszajn-Gagliardo-Slobodeckij seminorm
\begin{equation} \label{eq:Gagliardo-seminorm}
|v|_{H^\sigma(\Omega)} :=  \left( \frac{C(d,\sigma)}{2} \iint_{\Omega\times\Omega} \frac{|v(x)-v(y)|^2}{|x-y|^{d+2\sigma}} \, dx \, dy \right)^{1/2},
\end{equation}
and $C(d,\sigma)$ is the constant from \eqref{eq:defofLaps}.
We furnish this space with the norm 
\[
\|\cdot\|_{H^\sigma(\Omega)} := \left(\|\cdot\|_{L^2(\Omega)}^2 + |\cdot|_{H^\sigma(\Omega)}^2\right)^{1/2}, 
\]
and denote $( \cdot, \cdot )_{H^\sigma(\Omega)}$ the bilinear form
\[
( u , v )_{H^\sigma(\Omega)} := \frac{C(d,\sigma)}{2} \iint_{\Omega\times\Omega} \frac{(u(x)-u(y))(v(x)-v(y))}{|x-y|^{d+2\sigma}} \, dx \, dy, \quad u, v \in H^\sigma(\Omega).
\]
Moreover, if $\sigma\in(1,2)$ we set $H^\sigma(\Omega):= \big\{v\in H^1(\Omega): |v|_{H^{\sigma-1}(\Omega)}<\infty \big\}$.
\end{definition}

\begin{remark}[integrability $p$]
One could also define seminorms \eqref{eq:Gagliardo-seminorm} with differentiability $0<\sigma<2$ but integrability $p \in[1,\infty]$, in which case such spaces are denoted $W^\sigma_p(\Omega)$. In turn, the letter $H$ is often used to denote Bessel potential spaces $H^\sigma_p(\Omega)$; in case $p=2$ the spaces $W^\sigma_2(\Omega)$ and $H^\sigma_2(\Omega)$ coincide, and thus the notation $H^\sigma(\Omega)$ is typically employed.
\end{remark}

Of special interest to us are spaces consisting of zero-extension functions, namely for $\sigma\in(0,1]$
\[
\widetilde{H}^\sigma(\Omega) := \big\{ v \in H^\sigma(\mathbb{R}^d) \colon \supp v \subset \overline{\Omega} \big\};
\]
we define similarly $\widetilde{W}^\sigma_p(\Omega)$ and set $\Wzop(\Omega):=\widetilde{W}^1_p(\Omega)$ for $p\in[1,\infty]$. For these spaces, fractional seminorms are in turn norms. This is a consequence of the following well-known result.

\begin{lemma}[Poincar\'e inequality]
Let $\sigma \in (0,1)$, $p\in[1,\infty]$, and $\Omega$ be a bounded measurable domain. There is a constant $c=c(\Omega,d,\sigma,p)$ such that 
\[
\| v \|_{L^p(\Omega)} \leq c |v|_{W^\sigma_p(\polR^d)}  \quad \forall v \in  \widetilde{W}^\sigma_p(\Omega) .
\]
\end{lemma}
Therefore, in the case of our interest $p=2$,
\[
\| v \|_{\widetilde H^\sigma(\Omega)} := | v |_{H^\sigma(\mathbb{R}^d)}
\]
defines a norm equivalent to $\| \cdot \|_{H^\sigma(\mathbb{R}^d)}$ in $\widetilde{H}^\sigma(\Omega)$.

We define Besov spaces through real interpolation, following \cite{mclean2000strongly}. Given a compatible pair of Banach spaces $(X_0, X_1)$, $u \in X_0 + X_1$, and $t>0$, we set the $K$-functional
\begin{equation} \label{eq:K-functional}
K(t, u) = \inf \left\{ \left(\| u_0 \|_{X_0}^2 + t^2 \| u_1 \|_{X_1}^2 \right)^{1/2} \colon u = u_0 + u_1, \ u_0 \in X_0, \ u_1 \in X_1 \right\}. 
\end{equation}
For $\theta \in (0,1)$ and $q \in [1, \infty]$, let us define interpolation spaces
\[
(X_0, X_1)_{\theta, q} := \{ u \in X_0 + X_1 \colon \| u \|_{(X_0, X_1)_{\theta, q}} < \infty \},
\]
where
\begin{equation}\label{eq:Besov-norm}
\| u \|_{(X_0, X_1)_{\theta, q}} = \left\lbrace
\begin{aligned}
& \left[ q \theta (1-\theta) \int_0^\infty t^{-(1+\theta q)} |K(t,u)|^q \, dt \right]^{1/q} & \mbox{if } 1 \le q < \infty, \\
& \sup_{t > 0} \ t^{-\theta} |K(t,u)|  & \mbox{if } q = \infty.
\end{aligned}
\right.
\end{equation}
The normalization factor $q \theta (1-\theta)$ in the norm \eqref{eq:Besov-norm} guarantees the correct scalings in the limits $\theta \to 0$, $\theta \to 1$ and $q \to \infty$ for norm continuity.

\begin{definition}[Besov spaces]
Given $\sigma \in (0,2)$, $p,q \in [1, \infty]$, we define the spaces
\[ \begin{aligned}
& B^{\sigma}_{p,q} (\Omega) := \big(L^p(\Omega), W^2_p(\Omega) \big)_{\sigma/2,q}, \\
& \dot{B}^{\sigma}_{p,q} (\Omega) := \big\{ v \in B^{\sigma}_{p,q} (\Rd) \colon {\rm supp} \, v \subset \overline{\Omega} \big\}.
\end{aligned}\]
For $\sigma \in (0,1)$, the equivalent definition $\dot{B}^\sigma_{p,q} (\Omega) = \big(L^p(\Omega), \Wzop(\Omega)\big)_{\sigma,q}$ is valid. 
Moreover, we define
\begin{equation} \label{eq:def-negative-Besov}
B^{-\sigma}_{p,q} (\Omega) := \big(L^p(\Omega), W^{-1}_p(\Omega) \big)_{\sigma,q}.
\end{equation}
\end{definition}

Importantly, whenever $p=q$, Besov spaces reduce to Sobolev spaces \cite[\S\S 35--36]{Tartar07}:
\[
B^\sigma_{p,p} (\Omega) = W^\sigma_p(\Omega), \quad \dot{B}^\sigma_{p,p} (\Omega) = \widetilde W^\sigma_p(\Omega).
\]
There are two basic properties of Besov spaces on Lipschitz domains that will be useful in the sequel (cf. \cite[\S3.2.4, \S3.3.1]{Triebel10}: 
\[ \begin{array}{llll}
B^\sigma_{p,q_0}(\Omega) \subset B^\sigma_{p,q_1}(\Omega),  & \mbox{ if } \sigma > 0, &  1 \le p \le \infty, & 1 \le q_0 \le q_1 \le \infty; \\
B^{\sigma_1}_{p,q_1}(\Omega) \subset B^{\sigma_0}_{p,q_0}(\Omega) & \mbox{ if } 0 < \sigma_0 < \sigma_1, &   1 \le p \le \infty , & 1 \le q_0, q_1 \le \infty.
\end{array} \]
\red{In particular, for all $0 < \sigma_0 < \sigma_1$ it holds that
\begin{equation*} \label{eq:embedding_Besov}
B^{\sigma_1}_{2,\infty}(\Omega) \subset B^{\sigma_0}_{2,2}(\Omega) = H^{\sigma_0}(\Omega).
\end{equation*} }

We also have the following result regarding interpolation of Besov spaces (cf. \cite[Theorem 6.4.5]{BerghLofstrom}): given $\sigma_0 \ne \sigma_1$, $1 \le p, q_0, q_1, r \le \infty$ and $0 < \theta < 1$,
\begin{equation} \label{eq:interpolation_Besov}
\left( B^{\sigma_0}_{p,q_0}(\Omega), B^{\sigma_1}_{p,q_1}(\Omega) \right)_{\theta, r} = B^{\sigma}_{p,r}(\Omega) , \quad \mbox{where } \sigma = (1-\theta) \sigma_0 + \theta \sigma_1 .
\end{equation}
In particular, Besov spaces $B^{\sigma}_{p,q}(\Omega)$ with $\sigma\in(0,1)$ could be defined by interpolation either between $L^p(\Omega)$ and $W^2_p(\Omega)$ with index $\theta=\sigma/2$ or between $L^p(\Omega)$ and $W^1_p(\Omega)$ with index $\theta=\sigma$. Even though the spaces coincide, their norms defined in \eqref{eq:Besov-norm} are scaled differently. The corresponding factors $\big( q \sigma (1-\sigma/2) \big)^{1/q}$ and $\big( q \sigma (1-\sigma) \big)^{1/q}$ tend to zero as $\sigma\to2$ and $\sigma\to1$, respectively. \jp{Moreover, one can characterize spaces $B^{\sigma}_{p,q}(\Omega)$ with differentiability order $|\sigma| < 1/2$ through interpolation between negative and positive-order Sobolev spaces,
\begin{equation} \label{eq:Besov-intermediate}
\dot{B}^{\sigma}_{p,q}(\Omega) = 
\left( \Wzop(\Omega), W^{-1}_p(\Omega) \right)_{\theta, q}, \quad \mbox{with } \theta = \frac{1-\sigma}2 \in (1/4, 3/4).
\end{equation}
This characterization yields robust norms with respect to $\sigma$, and will be useful in Lemma \ref{lemma:embedding-Sobolev-Besov}.
}

The positive-order Besov spaces $\dot{B}^\sigma_{p,q}$ can be regarded as duals of negative-order Besov spaces. In fact, if $\sigma \in (0,1)$ and $p,q \in (1, \infty]$, by combining the property $\big((X_0,X_1)_{\sigma,q'}\big)' = (X_0', X_1')_{\sigma,q}$ (cf. \cite[Lemma 41.3]{Tartar07}) with definition \eqref{eq:def-negative-Besov} and the duality $\big(W^{-1}_{p'}(\Omega)\big)' = \Wzop(\Omega)$, we deduce
\begin{equation} \label{eq:dual-Besov}
\dot{B}^\sigma_{p,q} (\Omega) = \left( B^{-\sigma}_{p',q'} (\Omega) \right)'.
\end{equation}

\rhn{
We will need to relate Besov and Sobolev spaces. The following embedding is well known, but we include a simple proof to exhibit the explicit blow up of the continuity constant.
  
\begin{lemma}[embeddings between Besov and Sobolev spaces] \label{lemma:embedding-Sobolev-Besov}
Let $r > 0$ and $0<\epsilon <1/4$. Then, $B^{r}_{2,\infty}(\Omega) \subset H^{r-\epsilon}(\Omega)$ and
\begin{equation} \label{eq:embedding}
\|v\|_{H^{r-\epsilon}(\Omega)}  \lesssim \frac{1}{\sqrt{\epsilon}} \| v \|_{B^{r}_{2,\infty}(\Omega)} \quad \forall v \in B^{r}_{2,\infty}(\Omega).
\end{equation}
In addition, if $r\in(0,1/2)$, then
\begin{equation} \label{eq:embedding-negative}
\| v \|_{B^{-r}_{2,1}(\Omega)} \lesssim \frac{1}{\sqrt{r}} \|v\|_{L^2(\Omega)}  \quad \forall v \in L^2(\Omega).
\end{equation}
\end{lemma}
\begin{proof}
We employ the $K$-functional \eqref{eq:K-functional}. To prove \eqref{eq:embedding}, we let $k\ge0$ be the integer such that $r \in (k-1/2, k+1/2]$. We regard $H^{r-\epsilon}(\Omega)$ and $B^{r}_{2,\infty}(\Omega)$ as interpolation spaces between $X_0=H^{k-1}(\Omega)$ and $X_1=H^{k+1}(\Omega)$ with 
\[
H^{r-\epsilon}(\Omega) = \big[ H^{k-1}(\Omega),H^{k+1}(\Omega) \big]_{\sigma,2},
\quad
B^{r}_{2,\infty}(\Omega) = \big[ H^{k-1}(\Omega),H^{k+1}(\Omega) \big]_{\theta,\infty},
\]
where $\theta = \frac{r-k+1}{2}$ and $\sigma = \frac{r-\epsilon-k+1}{2}=\theta - \frac{\epsilon}2$.
This choice of spaces $(X_0,X_1)$ guarantees that $\theta \in [1/4,3/4]$ and $\sigma \in [1/8,3/4]$, if $\epsilon \in (0,1/4)$, are uniformly far from $0,1$ and the norms in \eqref{eq:Besov-norm} are robust. Given $v \in B^{r}_{2,\infty}(\Omega)$, using \eqref{eq:Besov-norm} for $q=2$ we deduce that for any $N\ge1$ to be found
\[
\|v\|_{H^{r-\epsilon}(\Omega)}^2 \lesssim \int_0^N t^{-(1+2\sigma)} \big| K(t,v) \big|^2 dt + \int_N^\infty t^{-(1+2\sigma)} \big| K(t,v) \big|^2 dt.
\]
Moreover, exploiting again \eqref{eq:Besov-norm} but now for $q=\infty$ yields
\[
\int_0^N t^{-(1+2\sigma)} \big| K(t,v) \big|^2 dt \le \sup_{t>0} \Big( t^{-2\theta} \big| K(t,v) \big|^2 \Big)
\int_0^N t^{-1+\epsilon}dt = \frac{N^{\epsilon}}{\epsilon} \|v\|_{B^{r}_{2,\infty}(\Omega)}^2.
\]
On the other hand, we clearly have $\big| K(t,v) \big| \le \|v\|_{H^{k-1}(\Omega)} \le C(\Omega) \|v\|_{H^{r-\epsilon}(\Omega)}$ and
\[
\int_N^\infty t^{-(1+2\sigma)} \big| K(t,v) \big|^2 dt \le C(\Omega)^2  \|v\|_{H^{r-\epsilon}(\Omega)}^2 \int_N^\infty t^{-(1+2\sigma)} dt
= \frac{C(\Omega)^2}{2 \sigma N^{2\sigma}}  \|v\|_{H^{r-\epsilon}(\Omega)}^2.
\]
Recalling that $2\sigma \in [1/2, 3/2]$ and choosing $N$ sufficiently large so that $\frac{C(\Omega)^2}{2 \sigma N^{2\sigma}} \le \frac12$ leads to the desired estimate \eqref{eq:embedding}. 

To prove \eqref{eq:embedding-negative}, we use a very similar technique; indeed, \jp{we exploit \eqref{eq:Besov-intermediate} to write}
\[
B^{-r}_{2,1}(\Omega) = (H^{-1}(\Omega), H^1(\Omega))_{\theta,1} \quad 
L^2(\Omega) = (H^{-1}(\Omega), H^1(\Omega))_{1/2,2},
\]
with $\theta = \frac{1-r}{2} \in (1/4,1/2)$. Given $v \in L^2(\Omega)$, an application of H\"older's inequality gives
\[
\int_{0}^1 t^{-(1+\theta)} |K(t,v)| dt
\jp{\le \left(\int_{0}^1 t^{-2} |K(t,v)|^2 dt \right)^{\frac12} \left( \int_{0}^1 t^{-1+r} dt\right)^{\frac12}}
\le \frac{1}{\sqrt{r}} \|v\|_{L^2(\Omega)} ,
\]
while the bound $K(t,v) \le \| u \|_{H^{-1}(\Omega)} \le C(\Omega) \| v \|_{L^2(\Omega)}$ leads to
\[
\int_{1}^\infty t^{-(1+\theta)} |K(t,v)| dt \le \frac{C(\Omega)}{\theta} \|v\|_{L^2(\Omega)} .
\]
Estimate \eqref{eq:embedding-negative} follows immediately from these two bounds. This concludes the proof.
\end{proof}
}
  
\subsection{Difference quotients in balls}\label{S:diff-balls}
We now characterize Besov spaces by means of first and second differences on balls. 
Given $\lambda >0$, we define the auxiliary domains
\begin{equation*} \label{eq:def_Omega_lambda}
  \Omega_\lambda = \{ x \in \Omega \colon d(x,\partial\Omega) > \lambda \},
  \quad
  \Omega^\lambda = \{ x \in \Rd \colon d(x,\partial\Omega) < \lambda \}.
\end{equation*}
Let $D = D_\rho(0)$ be the ball of radius $\rho$ centered at $0$. Given a function
$v\in L^p(\Omega)$ and direction $h\in D$, we consider the translation
$\tau(h) v(x) = v_h(x) := v(x+h)$ and first-order and second-order difference operators
$\delta_1(h)$ and $\delta_2(h)$ defined by
\begin{equation} \label{eq:translation-h-global}
  \delta_1(h) v(x) = v_h(x) - v(x),\quad
  \delta_2(h) v(x) = v_h(x) - 2v(x) + v_{-h}(x)
\end{equation}
for all $x \in \Omega_\rho$.

Besov semi-norms may be equivalently defined through difference quotients.
Since we are interested in fractional differentiability order $0<\sigma<2$,
we use second order differences to define the seminorms. For $p,q \in [1, \infty)$ we set
\begin{equation} \label{eq:centered-Besov-q}
[v]_{B^\sigma_{p,q}(\Omega)} := \left( q \sigma (2-\sigma) \int_{D} 
\frac{\| \delta_2(h) v \|^q_{L^p(\Omega_{|h|})}}{|h|^{d+ q \sigma}} dh
\right)^{1/q} ,
\end{equation}
while if $q = \infty$,
\begin{equation} \label{eq:centered-Besov}
[v]_{B^\sigma_{p,\infty}(\Omega)} := \sup_{h \in D} 
\frac{\| \delta_2(h) v \|_{L^p(\Omega_{|h|})}}{|h|^{\sigma}}.
\end{equation}
The scaling factor in \eqref{eq:centered-Besov-q} agrees with that in \eqref{eq:Besov-norm} for $\theta=\sigma/2$.
We emphasize that, even though in both \eqref{eq:centered-Besov-q} and \eqref{eq:centered-Besov} the norms depend on the radius $\rho$ of the ball $D$, the resulting seminorms are all equivalent.

The following result is classical in $\Rd$ \cite[Theorem 7.47]{AdamsFournier:2003}, and for Lipschitz domains one
can argue by using extension operators on Besov spaces (see \cite[Theorem 1, p. 381]{Nikolskii}).

\begin{lemma}[equivalence of Besov seminorms] \label{lemma:Besov}
Let $\Omega$ be a bounded Lipschitz domain, $\sigma \in (0,2)$, $p\in [1,\infty)$, $q \in [1,\infty]$ and $v \in L^p(\Omega)$. Then, the seminorm equivalence
\begin{gather*}
   |v|_{B^{\sigma}_{p,q}(\Omega)} \simeq [v]_{B^{\sigma}_{p,q}(\Omega)}
\end{gather*}      
is valid with constants that do not depend on $\sigma,p,q$.
\end{lemma}

We state an auxiliary result whose proof follows by interpolation between the trivial case $\sigma = 0$ (i.e. $L^p(\Omega)$) and the standard case $\sigma = 1$ (i.e. $W^1_p(\Omega)$).

\begin{lemma}[error estimate]\label{lemma:shift} Let \red{$p \in [1, \infty]$, $\sigma \in [0,1]$,} and $h \in D$. There exists $C > 0$ such that for any Lipschitz domain $\omega \subset \Rd$,
the translation operator $\delta_1(h)$ defined in \eqref{eq:translation-h-global} satisfies
\[
\| v - v_h \|_{L^p(\omega)} \le C |h|^\sigma \red{|v|_{W^\sigma_{p}(\omega^{|h|})} \quad \forall v \in W^\sigma_{p}(\omega^{|h|}).}
\]
\end{lemma}

Taking into account Lemma \ref{lemma:Besov} and Lemma \ref{lemma:shift}, it seems plausible to bound Besov seminorms by considering differences of fractional-order seminorms. This is the goal of the next proposition.

\begin{proposition}[reiteration of Besov seminorms] \label{prop:bound-Besov}
If $s \in (0,1)$, \red{$p \in [1, \infty]$, $\sigma \in [0,1]$,} 
and $\omega\subset\Rd$ is a Lipschitz domain, then
\[ \begin{split}
& |v|_{B^{s+\sigma}_{p,q}(\omega)} \lesssim \left( q (s+\sigma) (2-s-\sigma) \int_{D} \frac{\red{|v - v_h|^q_{W^s_{p}(\omega)}}}{|h|^{d + q\sigma}}  \, dh \right)^{1/q}, \quad q \in [1,\infty), \\ 
    & |v|_{B^{s+\sigma}_{p,\infty}(\omega)} \lesssim  \sup_{h \in D} \frac{\red{|v - v_h|_{W^s_{p}(\omega)}}}{|h|^{\sigma}} , \quad q = \infty.   
\end{split}\]
\end{proposition}
\begin{proof}
In view of Lemma \ref{lemma:Besov} and definitions \eqref{eq:centered-Besov-q} or \eqref{eq:centered-Besov},
if $t := \sigma + s < 2$ we infer that
\[
|v|_{B^{t}_{p,q}(\omega)} \simeq [v]_{B^t_{p,q}(\omega)} = \left\lbrace
\begin{aligned}
& \left( q t (2-t) \int_{D} 
\frac{\| v_h - 2v + v_{-h} \|^q_{L^p(\omega_{|h|})}}{|h|^{d+ q t}} dh
\right)^{1/q} & \mbox{ if } q \in [1,\infty), \\
& \sup_{D} 
\frac{\| v_h - 2v + v_{-h} \|_{L^p(\omega_{|h|})}}{|h|^t} & \mbox{ if } q = \infty.
\end{aligned}
\right.
\]
Letting $w = v_h - v$, we write
\[
\| v_h - 2v + v_{-h} \|_{L^p(\omega_{|h|})} = \| w - w_{-h}\|_{L^p(\omega_{|h|})}
\]
whence applying Lemma \ref{lemma:shift} (error estimate) to $w$ and observing that
$(\omega_{|h|})^{|h|} \subset \omega$ we obtain
\[
\| v_h - 2v + v_{-h} \|_{L^p(\omega_{|h|})} \lesssim |h|^s \red{|v - v_h|^q_{W^s_{p}(\omega)}}.
\]
This yields the asserted estimates.
\end{proof}

\rhn{The following estimate quantifies the precise blow-up of first differences $\delta_1(h)$ relative to second differences $\delta_2(h)$ in the definition of $\|\cdot\|_{{B}^\sigma_{p,q}(\Omega)}$ as $\sigma\to 1^-$. We state it now in the particular case $p=2$ and $q=\infty$, which is of interest later, but refer to \cite{Ditzian88} for a general statement and proof.

\begin{lemma}[Marchaud inequality]\label{L:Marchaud}
For all $\sigma \in (0,1)$ there holds
\begin{equation}\label{eq:Marchaud}
  \sup_{h\in D} \frac{\| \delta_1(h) v \|_{L^2(\Omega_{|h|})}}{|h|^{\sigma}} \lesssim
  \|v\|_{L^2(\Omega)} + 
  \frac{1}{\sqrt{1-\sigma}} \sup_{h\in D} \frac{\| \delta_2(h) v \|_{L^2(\Omega_{|h|})}}{|h|^{\sigma}}
  \qquad\forall v\in {B}^\sigma_{2,\infty}(\Omega),
\end{equation}  
\end{lemma}
}

\subsection{Localization of Besov norms}\label{S:localization}
We next show that Besov seminorms can be equivalently written as sums of norms over partitions, as long as the partitions have some overlap.

\begin{lemma}[localization]\label{L:localization}
Let $p, q \in [1, \infty]$ and $\sigma \in (0,2)$. Let $\{ D_j \}_{j = 1}^M$ be a finite covering of $\Omega$ by balls $D_j = D_\rho(x_j)$ of radius $\rho$ and center $x_j$. If $v\in L^p(\Omega)$, then $v \in B^\sigma_{p,q}(\Omega)$ if and only if $v \big|_{\Omega \cap D_j} \in B^\sigma_{p,q}(\Omega \cap D_j)$ for all $j = 1, \ldots, M$, and
\begin{equation} \label{eq:localization}
| v |_{B^\sigma_{p,q}(\Omega)}^p \simeq \sum_{j=1}^M | v |_{B^\sigma_{p,q}(\Omega \cap D_j)}^p.
\end{equation}
Moreover, fix $\delta \ge \rho$, consider a finite cover as above of $\Omega^\delta$, and let $v \colon \Rd \to \mathbb{R}$ satisfy $\rm{supp} \, v \subset \overline\Omega$. If $v\in L^p(\Omega)$, then $v \in \dot{B}^\sigma_{p,q}(\Omega)$ if and only if $v \big|_{D_j} \in B^\sigma_{p,q}(D_j)$ for all $j = 1, \ldots, M$, and
\begin{equation} \label{eq:localization2}
| v |_{\dot{B}^s_{p,q}(\Omega)}^p \simeq \sum_{j=1}^M | v |_{B^s_{p,q}(D_j)}^p.
\end{equation}
The equivalence constants above depend on $s,p,q,\Omega$ and the cover chosen.
\end{lemma}
\begin{proof}
The first assertion is a consequence of the equivalence $\| w \|_{L^p(\Omega)}^p \simeq \sum_{j=1}^M \| w \|_{L^p(\Omega \cap D_j)}^p$ 
applied to $w = v_h - 2v + v_{-h}$ and the equivalence of $\ell^q$ norms in $\mathbb{R}^M$.
For the second statement it suffices to realize that, since $\delta \ge \rho$, $\| v \|_{\dot{B}^s_{p,q}(\Omega)} = \| v \|_{B^s_{p,q}(\Rd)} = \| v \|_{B^s_{p,q}(\Omega^\delta)}$.
\end{proof}

  It is worth stressing the dependence of the equivalence constants in \eqref{eq:localization}
  and \eqref{eq:localization2} on the covering $\{B_j\}_{j=1}^M$. For integer-order Sobolev spaces
  $W^k_p(\Omega)$
  these constants only depend on the covering overlap but not on its cardinality $M$. For fractional
  Sobolev and Besov spaces the constants also depend on $M$ \cite{NguyenSickel:18}. However, in the
  arguments below $M$ is fixed.

\subsection{Difference quotients in cones}\label{S:diffs-cones}
%
Since translations $v_h$ of any $v\in\dot{B}^\sigma_{p,q}(\Omega)$ must belong to $\dot{B}^\sigma_{p,q}(\Omega)$
in the subsequent developments, we need to cope with three crucial questions. First, we must localize such translations, an issue we take over in Section \ref{S:local-translations}. Second, we must restrict the admissible set of directions $h$ from a ball $D_\rho(0)$ to cones to deal with the Lipschitz character of $\Omega$. Third, we must deal with second order differences within cones because $\sigma\in(0,2)$. We tackle the last two issues next.

\begin{definition}[generating set]\label{D:generating-set}
We say that a bounded set $D$ star-shaped with respect to the origin generates $\Rd$ if 
there exists $\rho_0(D) > 0$ such that for every $\rho \le \rho_0(D)$ and
every $h\in D_\rho(0)$, there exists $\{h_j\}_{j=1}^d \subset D\cup(-D)$ satisfying
\[
h = \sum_{j=1}^d h_j,
\quad
\sum_{j=1}^d |h_j| \le c |h|
\]
with a constant $c>0$ only dependent on $D$.
\end{definition}
An immediate property of generating sets is the following scaling invariance: if $D$ generates $\Rd$, $\rho_0(D)$ is as in the definition above, and $\lambda > 0$, then $\lambda D$ also generates $\Rd$ and $\rho_0 (\lambda D) = \lambda \rho_0(D)$.

Given $v\in L^p(\Rd)$, consider the translation $\tau(h) v = v_h$ and the first-order modulus of regularity
\[
\omega_1(h) = \omega_1(v,h) := \|\tau(h) v - v\|_{L^p(\Rd)} = \|\delta_1(h) v\|_{L^p(\Rd)}.
\]
The following elementary properties are valid:
\begin{itemize}
\item
  {\it Symmetry}: $\omega_1(-h) = \omega_1(h)$ (simply change variables $x\mapsto x+h$);

\item
  {\it Subadditivity}: $\omega_1(h_1+h_2) \le \omega_1(h_1) + \omega_1(h_2)$ (simply apply the triangle
  inequality $\|\tau(h_1+h_2) v-v\|_{L^p(\Omega)}\le \|\tau(h_1+h_2) v- \tau(h_2)v\|_{L^p(\Omega)} +
      \|\tau(h_2) v-v\|_{L^p(\Omega)}$.
\end{itemize}
Symmetry enables us to disregard the set $-D$ in Definition \ref{D:generating-set} and consider only $h\in D$ when computing $\omega_1(h)$. Subadditivity yields the following important relation.

\begin{lemma}[first-order difference quotient]\label{L:restrict-1}
  Let $D$ be a set that generates $\Rd$. Then, for $v\in L^p(\Rd)$ and $h \in D_\rho(0)$ with $h= \sum_{j=1}^d h_j$ as in Definition \ref{D:generating-set},
  \[
  \frac{\omega_1(h)}{|h|^\sigma} \le c^\sigma \sum_{j=1}^d \frac{\omega_1(h_j)}{|h_j|^\sigma}.
  \]
\end{lemma}
\begin{proof}
Simply apply subadditivity of $\omega_1$ in conjunction with Definition \ref{D:generating-set}.
\end{proof}

Lemma \ref{L:restrict-1} reveals that we can restrict the set of directions $h$ from the
ball $D_\rho(0)$ to a generating
set $D$ of $\Rd$ in the definition of Besov spaces of order $0<\sigma<1$ by using first-order difference
quotients. However, we need to extend this property to second-order difference quotients in light of
our definitions \eqref{eq:centered-Besov-q} and \eqref{eq:centered-Besov} and corresponding
modulus of regularity $\omega_2$ for $v\in L^p(\Rd)$:
\[
\omega_2(h) = \omega_2(v,h) := \|\tau(h)v - 2v + \tau(-h)v\|_{L^p(\Rd)} = \|\delta_2(h) v\|_{L^p(\Rd)} .
\]
In the same spirit of Lemma \ref{L:restrict-1}, we have to express $\omega_2(h)$ for any
$h\in D_\rho(0)$ in terms of $\omega_2(h)$ for $h\in D$, but we cannot longer ignore the orientation
of $h$ allowed by the symmetry of $\omega_1$. Let $h=\sum_{j=1}^d h_j \in D_\rho(0)$ be arbitrary
and decompose the set $\{h_j\}_{j=1}^d\subset D$ of Definition \ref{D:generating-set}
as follows:
\[
h_j \in D \quad 1\le j\le m,
\qquad
h_j \in -D \quad m< j \le d.
\]
We further assume that $D$ is a {\it convex cone} to deduce $\sum_{j=1}^m h_j \in D$
and $\sum_{j=m+1}^d h_j \in -D$.
Therefore, we have to be able to express $\omega_2(h_1-h_2)$ for arbitrary directions $h_1,h_2\in D$
in terms of admissible directions $h \in D$. We tackle this next.

\begin{lemma}[second-order difference quotients]\label{L:restrict-2}
  Let $D$ be a convex cone generating $\Rd$. If $h_1,h_2\in D/2$, then
  $2h_1, 2h_2, h_1+h_2 \in D$ and 
  \[
  \omega_2(h_1-h_2) \le \frac12 \Big( \omega_2(2h_1) + 2 \omega_2(h_1+h_2) + \omega_2(2h_2) \Big).
  \]
\end{lemma}
\begin{proof}
The assertion is a trivial consequence of the elementary relation
\[
2\delta_2(h_1-h_2)v = \delta_2(2h_1)v + \delta_2(2h_2)v - \delta_2(h_1+h_2)
\Big( \tau(h_1-h_2) v + \tau(h_2-h_1) v \Big),
\]
the definition of $\omega_2$ and the property $\|\tau(h)v\|_{L^p(\Rd)} = \|v\|_{L^p(\Rd)}$.
\end{proof}

Given a set $\Lambda \subset \Rd$, $p \in [1,\infty)$, $q = \infty$, and $\sigma \in (0,2)$, let $[\cdot ]_{B^\sigma_{p,\infty}(\Rd; \Lambda)}$ denote the Besov seminorm in \eqref{eq:centered-Besov}, computed by taking second-order differences over $\Lambda $. We next compare Besov seminorms for $\Lambda$ being either a ball or a cone.
  
\begin{proposition}[Besov seminorms using cones]
Let $D$ be a convex cone generating $\Rd$ so that $D \subset D_{\rho_1} = D_{\rho_1}(0)$ and let $\rho_0=\rho_0(D)$ be as in Definition \ref{D:generating-set}. Then, every function $v$ satisfies
\begin{equation} \label{eq:equivalence-Besov-cones}
\frac{1}{c^{\sigma}(2^{\sigma}+ 1)} [v]_{B^\sigma_{p,\infty}(\Rd; D_{\rho_0/2})} \le [v]_{B^\sigma_{p,\infty}(\Rd; D)} \le [v]_{B^\sigma_{p,\infty}(\Rd; D_{\rho_1})},
\end{equation}
where $c$ is the constant from Definition \ref{D:generating-set}.
Consequently, we have $[v]_{B^\sigma_{p,\infty}(\Rd; D)} \simeq |v|_{B^\sigma_{p,\infty}(\Rd)}$.
\end{proposition}
\begin{proof}
The fact that $[v]_{B^\sigma_{p,\infty}(\Rd; D)} \le [v]_{B^\sigma_{p,\infty}(\Rd; D_{\rho_1})}$ is an obvious consequence of the set inclusion $D \subset D_{\rho_1}$. Next, let us fix $h \in D_{\rho_0/2}$, and decompose it according to Definition \ref{D:generating-set}. Specifically, by the discussion preceding Lemma \ref{L:restrict-2} (second-order difference quotients), we write $h = h_1 - h_2$, with $h_1, h_2 \in D/2$ and $|h_1|+|h_2| \le c |h|$. We have
\[ \begin{aligned}
\frac{\omega_2(h)}{|h|^\sigma} & \le \frac1{2|h|^\sigma} \Big( \omega_2(2h_1) + 2 \omega_2(h_1+h_2) + \omega_2(2h_2) \Big) \\
& \le \frac{c^{\sigma}}{2} \Big( \frac{\omega_2(2h_1)}{|h_1|^\sigma} + \frac{2 \omega_2(h_1+h_2)}{|h_1 + h_2|^\sigma} + \frac{\omega_2(2h_2)}{|h_2|^\sigma} \Big).
\end{aligned} \]
Because $2h_1, h_1 + h_2, h_2 \in D$, we immediately obtain the upper bound
\[
\frac{\omega_2(h)}{|h|^\sigma} \le  \frac{c^{\sigma}}{2} (2^{\sigma + 1}+ 2) \sup_{\tilde h \in D} \frac{\omega_2(\tilde h)}{|\tilde h|^\sigma}
= c^{\sigma} (2^{\sigma}+ 1) \, [v]_{B^\sigma_{p,\infty}(\Rd; D)}.
\]
The first inequality in \eqref{eq:equivalence-Besov-cones} follows upon taking supremum over $h \in D_{\rho_0/2}$.

The second statement in the proposition is a consequence of Lemma \ref{lemma:Besov} and \eqref{eq:equivalence-Besov-cones}.
\end{proof}

\subsection{Localized translations and admissible directions}\label{S:local-translations}

Global translations such as \eqref{eq:translation-h-global} are not appropriate to capture the local behavior of solutions of \eqref{eq:Dirichlet}. Instead, we shall operate by localizing the translations and restricting the admissible directions $h$, as proposed by Savar\'e in \cite{Savare98}. Importantly, the former can be achieved by means of a convex combination between the identity operator and a translation, where the factor is related to a given cut-off function.

\begin{definition}[localized translation operator] \label{def:translation-operator}
For every function $v \colon \Omega \to \mathbb{R}$ we denote by $\tilde v$ its extension by zero outside $\Omega$ and, according to \eqref{eq:translation-h-global}, use the notation
\begin{equation*} \label{eq:translation-h}
v_h (x) = \tilde v (x + h) \quad x, h \in \Rd . 
\end{equation*}
Given $x_0$ and $\rho$, let $D_\rho(x_0)$ be the ball of radius $\rho$ and center $x_0$. We fix a cut-off function $\phi$ such that $0 \le \phi \le 1$, $\phi \equiv 1$ on $D_\rho (x_0)$, $\supp \phi \subset D _{2\rho}(x_0)$.
For those $x_0$ and $\rho$, given $h \in \Rd$, we define the localized translation operator
\begin{equation} \label{eq:translation-operator}
T_h v = \phi v_h + (1 - \phi) v .
\end{equation}
\end{definition}

\red{We now consider some variants of Lemma \ref{lemma:shift} by using this localized translation operator.  

\begin{lemma}[error estimate for $T_h$] \label{lemma:shift-Th}
Let $T_h$ be given according to Definition \ref{def:translation-operator}. Then, for every $h \in \Rd$, $\sigma\in[0,1]$,  $q\in[1,\infty]$, and $\gamma \in (-1,1)$ we have 
\begin{equation}\label{eq:translation-Besov}
  \| v - T_h v \|_{B^{\gamma}_{2,q}(D_{2\rho} (x_0))} \lesssim
  |h|^\sigma \|v\|_{B^{\gamma+\sigma}_{2,q}(D_{3\rho} (x_0))}
  \quad\forall \, v\in B^{\gamma+\sigma}_{2,q}(D_{3\rho} (x_0)).
\end{equation}
The hidden constant above is independent of $\sigma$ and $q$, but may blow up as $\gamma\to\pm1$.

Additionally, for every $\sigma \in (0,1)$ we have the estimate
\begin{equation} \label{eq:translation-L2}  
\| v - T_h v \|_{L^2(D_{2\rho}(x_0))} \lesssim \frac{|h|^\sigma}{\sqrt{1-\sigma}} \|v\|_{B^\sigma_{2,\infty}(D_{3\rho}(x_0))} \quad \forall v \in B^\sigma_{2,\infty}(D_{3\rho}(x_0)).
\end{equation}
\end{lemma}
\begin{proof}
Because $T_h v - v = \phi \, \delta_1(h) v$ and $\supp \phi \subset D _{2\rho}(x_0)$, we have the following estimates:
\[ \begin{split}
\| v- T_h v\|_{H^{-1}(D_{2\rho}(x_0))} \lesssim |h| \|v\|_{L^2(D_{3\rho}(x_0))} , \\
\| v- T_h v\|_{H^{1}(D_{2\rho}(x_0))} \lesssim |h| \|v\|_{H^2(D_{3\rho}(x_0))} .
\end{split}\]
Let $\gamma \in (-1,1)$ and $q \in [1, \infty]$. By interpolation (cf. \eqref{eq:interpolation_Besov}), we obtain
\[
\| v- T_h v\|_{B^{\gamma}_{2,q}(D_{2\rho}(x_0))} \lesssim |h| \|v\|_{B^{1+\gamma}_{2,q}(D_{3\rho}(x_0))} ,
\]
while
\[
\| v- T_h v\|_{B^{\gamma}_{2,q}(D_{2\rho}(x_0))} \lesssim \|v\|_{B^{\gamma}_{2,q}(D_{3\rho}(x_0))},
\]
is trivial. Thus, \eqref{eq:translation-Besov} follows by interpolating between the latter two estimates.

To prove \eqref{eq:translation-L2}, we again exploit the fact that $T_h v - v = \phi \, \delta_1(h) v$ and combine \eqref{eq:centered-Besov} with the Marchaud inequality \eqref{eq:Marchaud}.
\end{proof}
}


Having localized the translation, a missing key ingredient to handle the Lipchitz character of $\Omega$ is to determine a convex cone $D$ of admissible directions. For the operator $T_h$ in \eqref{eq:translation-operator}, this boils down to determining the vectors $h$ with respect to which translate and yet remain in $\Omega$.

\begin{definition}[admissible outward vectors]\label{D:admissible-vectors}
For every $x_0 \in \Rd$ and $\rho \in (0,1]$, we define the set of admissible outward vectors
\[
\calO_\rho(x_0) = \{ h \in \Rd \colon |h|\le \rho, (D_{3\rho}(x_0)\setminus\Omega) + th \subset \Omega^c, \ \forall t \in [0,1] \}.
\]
\end{definition}

The set $\calO_\rho(x_0)$ gives the admissible translations in the sense that, if, given $x_0 \in \Omega$ and $\rho \in (0,1]$, we fix $\phi$ and define $T_h$ according to \eqref{eq:translation-operator} for some $h \in \calO_\rho(x_0)$, then $T_h v \in \dot{B}^\sigma_{p,q}(\Omega)$ for all $v \in \dot{B}^\sigma_{p,q}(\Omega)$. Indeed, it is clear from its definition that if $v \in B^\sigma_{p,q}(\Rd)$ then $T_h v \in B^\sigma_{p,q}(\Rd)$ for all $h \in \Rd$; moreover, if $h \in \calO_\rho(x_0)$, for a.e $x \in \Omega^c$ we have
\[ \begin{split}
& x \in \Omega^c \cap D_{3\rho}(x_0) \Rightarrow x+h \in \Omega^c \Rightarrow T_h v (x) = \phi(x) v (x+h) + (1 - \phi(x)) v(x) = 0, \\
& x \in \Omega^c \setminus D_{3\rho}(x_0) \Rightarrow \phi(x) = 0 \Rightarrow T_h v (x) = v(x) = 0.
\end{split} \]

We now define the admissible convex cone $C(x_0)$ for each $x_0\in\Rd$. We rely on the following
uniform cone property satisfied by bounded Lipschitz domains;
we refer the reader to \cite[\S 1.2.2]{Grisvard}.

\begin{proposition}[uniform cone property] \label{prop:cone-property}
If $\Omega$ is a bounded Lipschitz domain, then there exist $\rho \in (0,1]$, $\theta \in (0, \pi]$ and a map $\vn \colon \Rd \to S^{d-1}$ such that, for every $x_0 \in \Rd$,
\[
\calC_\rho (\vn(x_0), \theta) := \{ h \in \Rd \colon  |h| \le \rho, \ h \cdot \vn(x_0) \ge |h| \cos \theta \} 
\subset \calO_\rho(x_0).
\]
\end{proposition}

An obvious consequence of Proposition \ref{prop:cone-property} for bounded Lipschitz domains is that all directions within $\calC_\rho (\vn(x_0), \theta)$ are admissible outward vectors starting from any point $x_0 \in \Rd$. Moreover, $\calC_\rho (\vn(x_0), \theta)$ is symmetric with respect to its axis $\vn(x_0)$ and with fixed opening $\theta$, whence
\[
  C(x_0) = C_\rho(x_0) := \calC_\rho (\vn(x_0), \theta)
\]
is a convex cone generating $\Rd$ according to Definition \ref{D:generating-set}. In view of the discussion in Section \ref{S:diffs-cones}, we can replace the ball $D=D_\rho(0)$ by $C_\rho(x_0)$ in the definition of Besov seminorms \eqref{eq:centered-Besov-q} or \eqref{eq:centered-Besov} thereby retaining equivalent quantities. We exploit the cone  $C_\rho(x_0)$ in Section \ref{sec:regularity} to construct suitable test functions.

\subsection{Regularity of functionals}\label{S:functional-reg}

Inspired by \cite[Formula (13)]{Savare98}, we introduce a notion of regularity of functionals that measures their sensitivity  with respect to a family of perturbations.

\begin{definition}[modulus of $(T,D,\gamma)$-regularity]\label{def:regularity}
Let $V$ be a Banach space, $K \subset V$ and $\gamma >0$. Given a family of maps $T_h \colon K \to K$, with $h$ varying on a given set $D \subset \Rd,$ we say that a functional $\calF:V\to\polR$ is $(T,D,\gamma)$-regular on $K$ if, for all $v \in K$,
\[
\omega(v) = \omega(v; \calF, T, D, \gamma) := \sup_{h \in D }
\frac{\big| \calF (T_h v) - \calF (v) \big|}{|h|^\gamma} < \infty .
\]
\end{definition}

\begin{remark}[subadditivity]
The modulus of $(T,D,\gamma)$-regularity is subadditive with respect to the $\calF$-argument. Indeed, we have that
\begin{equation} \label{eq:subadditivity}
\omega(v; \calF_1 + \calF_2, T, D, \gamma) \le \omega(v; \calF_1, T, D, \gamma)
+ \omega(v; \calF_2, T, D, \gamma).
\end{equation}
Thus, in order to prove the $(T,D,\gamma)$-regularity of $\calF_1 + \calF_2$, it suffices to show the regularity of each of the two functionals independently.
\end{remark}

To prove Besov regularity estimates of solutions to \eqref{eq:Dirichlet}, we will determine separately the regularity of the maps $\calF_1$, $\calF_2$ in \eqref{eq:def-subfunctionals} with respect to the family of local translations \eqref{def:translation-operator}.
The next lemma shows that one can derive Besov regularity estimates for minimizers of \eqref{eq:def-functional} by proving that $\calF$ is regular in the sense of Definition \ref{def:regularity}. This clever idea goes back to Savar\'e \cite{Savare98}.

\begin{lemma}[regularity and minimizers] \label{lem:corollary1}
Let $x_0 \in \Rd$, $\rho > 0$ and $h \in C_\rho(x_0)$. Consider translation operators $T_h \colon \tHs \to \tHs$ as in \eqref{eq:translation-operator}.
If $u$ solves \eqref{eq:Dirichlet} and the functional $\calF$ defined in \eqref{eq:def-functional} is
$(T,C(x_0),\gamma)$-regular on $\tHs$ for some $\gamma >0$, then
\[
|u-T_h u|_{H^s(\Rd)}^2 \le 2 \omega(u) |h|^\gamma.
\]
\end{lemma}
\begin{proof}
The proof follows immediately from Definition \ref{def:regularity} and the fact that the minimizer $u$ verifies
\[
\calF(v) - \calF(u) = \langle\delta\calF(u),v\rangle + \frac12 |u - v|_{H^s(\Rd)}^2 = \frac12 |u - v|_{H^s(\Rd)}^2 \quad \forall v \in \tHs,
\]
because $\calF$ is quadratic and the first variation $\delta\calF(u)$ of $\calF$ at $u$ vanishes.
\end{proof}

\section{Regularity of the functionals} \label{sec:regularity}

In this section we study separately the local regularity of the functionals $\calF_1$ and $\calF_2$ in the sense of Definition \ref{def:regularity}. To this end, we choose an arbitrary point $x_0\in\Rd$ and denote
\[
C_\rho = C_\rho(x_0), \quad D_\rho = D_\rho(x_0)
\]
cones and balls centered at $x_0$, respectively.
We exploit the key property that $T_hv\in\tHs$ for any $v\in\tHs$ and $h\in C_\rho$. This enables us to evaluate $\calF(T_hv)$ and compare it with $\calF(v)$.

\subsection{Regularity of $\calF_1$}
We start with the linear functional $\calF_1$ \rhn{defined in \eqref{eq:def-subfunctionals}.} As expected, the smoothness of the right hand side $f$ plays a key role in the regularity of $\calF_1$.

\begin{proposition}[regularity of $\calF_1$] \label{prop:regularity-f}
\red{Let $\sigma \in (0,1)$} and $f \in L^2(\Omega)$. Let $T_h$ be the translation operator given in \eqref{eq:translation-operator}. Then, $\calF_1$ is $(T, C_\rho,\sigma)$-regular in $\dot{B}^{\sigma}_{2,\infty}(\Omega)$ for all $x_0$, $\rho$, namely,
\begin{equation}\label{eq:regularity-f}
  \sup_{h \in C_\rho } \frac{\calF_1 (T_h v) - \calF_1 (v)}{|h|^\sigma} \lesssim
 \red{\frac{1}{\sqrt{1-\sigma}} \| f \|_{L^2(D_{2\rho} \cap \Omega)}  \| v \|_{B^\sigma_{2,\infty}(D_{3\rho})},}
  \quad \forall v \in \dot{B}^{\sigma}_{2,\infty}(\Omega).
\end{equation}
\red{Additionally, if $q \in (1,\infty]$, $\sigma \in [0,1]$ and $f \in B^{-\gamma}_{2,q'}(\Omega)$ for some $\gamma \in (0,1)$, where $q' = \frac{q}{q-1}$ is the H\"older conjugate of $q$, then 
$\calF_1$ is $(T, C_\rho,\sigma)$-regular in $\dot{B}^{\sigma+\gamma}_{2,q}(\Omega)$ for all $x_0$, $\rho$, and
\begin{equation}\label{eq:regularity-f-Besov}
\sup_{h \in C_\rho } \frac{\calF_1 (T_h v) - \calF_1 (v)}{|h|^{\sigma}} \lesssim \| f \|_{B^{-\gamma}_{2,q'}(\Omega)} \red{\| v \|_{B^{\sigma+\gamma}_{2,q}(D_{3\rho})}}, \quad \forall v \in \dot{B}^{\sigma+\gamma}_{2,q}(\Omega).
\end{equation}
}
\end{proposition}
\begin{proof}
First, assume $f \in L^2(\Omega)$.
Because $\calF_1$ is linear and $T_h v - v = \phi (v_h - v)$, we have
\[
\calF_1 (T_h v) - \calF_1 (v) = \int_\Omega f \phi (v_h - v) \le
\| f \|_{L^2(D_{2\rho} \cap \Omega )}  \| T_h v - v \|_{L^2(D_{2\rho})}.
\]
We point out that, although $D_{2\rho}$ may have nonempty intersection with $\Omega^c$, $v_h-v$ equals zero on $\Omega^c$ because both $v$ and $v_h$ vanish for any admissible direction $h$.
\red{Applying \eqref{eq:translation-L2}  to $v \in B^\sigma_{2,\infty}(D_{3\rho})$ with $|h| \le \rho$, we obtain
\[
\calF_1 (T_h v) - \calF_1 (v) \lesssim \frac{|h|^\sigma}{\sqrt{1-\sigma}}
\| f \|_{L^2(D_{2\rho} \cap \Omega)}  \| v \|_{B^\sigma_{2,\infty}(D_{3\rho})}.
\]
}
This establishes \eqref{eq:regularity-f}. To prove \eqref{eq:regularity-f-Besov}, we assume \red{$f \in B^{-\gamma}_{2,q'}(\Omega)$ and immediately deduce by the duality property \eqref{eq:dual-Besov}
\[
\calF_1 (T_h v) - \calF_1 (v) \le \| f \|_{B^{-\gamma}_{2,q'}(\Omega)}  \| T_h v - v \|_{\dot{B}^\gamma_{2,q}(\Omega)}.
\] 
Thus, because $T_h v - v $ vanishes in $D_{2\rho}^c$, \eqref{eq:regularity-f-Besov} follows by applying \eqref{eq:translation-Besov}.}
\end{proof}

\begin{remark}[more regular functions] 
In the application of \eqref{eq:regularity-f} or \eqref{eq:regularity-f-Besov} to problem \eqref{eq:Dirichlet}, the natural choice is $\sigma = s$. Even though a priori it is not clear that the solution $u$ to such a problem is any more regular than $H^s$, once one is able to show certain regularity, then one can revisit either of these estimates to deduce a higher-order regularity estimate for the functional $\calF_1$. \red{In principle,} this process can be iterated until one hits the \red{maximum} value $\sigma = 1$ that
translates into a regularity pickup of order $1/2$. This bootstrapping argument is developed in Section \ref{sec:Besov-regularity}.
\end{remark}

\begin{remark}[localization]
  We point out that we deliberately did not localize the estimate \eqref{eq:regularity-f-Besov}. This is
  because localizing \rhn{$\| f \|_{B^{-\gamma}_{2,q'}(\Omega)}$} would require dealing with Lemma \ref{L:localization}
  (localization) for seminorms in the Besov space \rhn{$\dot{B}^\gamma_{2,q}(\Omega)$}, and specifically with
  \eqref{eq:localization2}. Since the equivalence constant in \eqref{eq:localization2} depends on $M$,
  localizing \rhn{$\| f \|_{B^{-\gamma}_{2,q'}(\Omega)}$} would not remove the sensitivity on $M$.
\end{remark}

\subsection{Regularity of $\calF_2$}
We next discuss the regularity of the quadratic functional $\calF_2$ \rhn{defined in \eqref{eq:def-subfunctionals}.} To this end, we introduce an
unusual semi-local fractional seminorm for all $r>0$
\[
|v|_{H^s(D_r,\Rd)} := \bigg(\int_{D_r}\int_{\Rd} \frac{\big( v(x) - v(y) \big)^2}{|x-y|^{d+2s}} dx dy
\bigg)^{\frac12};
\]
note the accumulation property $\sum_j |v|_{H^s(D_r(x_j),\Rd)}^2 \simeq |v|_{H^s(\Rd)}^2$ if
the covering $\{D_r(x_j)\}_j$ has finite overlap.
This seminorm may of course be replaced by the more elegant global fractional norm $|v|_{H^s(\Rd)}$, but
the accumulation property will be used in the proof of Theorem \ref{thm:regularity-solutions-rep} below.

\begin{proposition}[regularity of $\calF_2$] \label{prop:regularity-F0}
Let $\sigma \in [s,1]$ and $q \in [1, \infty]$, with the condition $q \le 2$ if $\sigma = s$. Let the translation operator $T_h$ and cut-off function $\phi$ obey \eqref{eq:translation-operator}. The functional $\calF_2 \colon \tHs \to \tHs$ given by $\calF_2 (v) = \frac12 |v|_{H^s(\Rd)}^2$ is $(T,C_\rho,\sigma)$-regular in $\dot{B}^\sigma_{2,q}(\Omega)$ for all $x_0$, $\rho$, namely,
\begin{equation} \label{eq:higher-regularity-F0}
\sup_{h \in C_\rho} \frac{\calF_2 (T_h v) - \calF_2 (v)}{|h|^\sigma} \lesssim \|\phi\|_{W^1_\infty(\Rd)} |v|_{B^\sigma_{2,q}(D_{4\rho})} |v|_{H^s(D_{4\rho},\Rd)}.
\end{equation}
\end{proposition}
\begin{proof}
We first observe that $\dot{B}^\sigma_{2,q}(\Omega) \subset \tHs$ under the 
  the hypotheses on $\sigma$ and $q$: for all $q \in [1,\infty]$ if $\sigma > s$, whereas
  for all $q\le 2$ for $\sigma=s$. We proceed in several steps.

\medskip
{\it Step 1: Auxiliary function $z$.}
In light of the definition \eqref{eq:translation-operator} of $T_h$
\[
T_h v(x) = \phi(x) v_h(x) + (1-\phi(x)) v(x),
\]
we rewrite the difference $T_hv(x) - T_hv(y)$ as follows
\begin{equation}\label{z1-z2}
  T_hv(x) - T_hv(y) = z(x,y) + \big( \phi(x) - \phi(y) \big) \big( v_h(y) - v(y) \big)
\end{equation}
in terms of the auxiliary function
\begin{align*}
  z(x,y) := \phi(x) \big(v_h(x) - v_h(y) \big) + (1-\phi(x))  \big( v(x) - v(y) \big).
\end{align*}

\medskip
{\it Step 2: Decomposition of $\calF_2(T_h v) - \calF_2(v)$.}
Since $\supp \phi \subset D_{2\rho}$, we deduce
\begin{equation}\label{Thv=v}
T_h v(x) = v(x) \quad\forall \, x\notin D_{2\rho}.
\end{equation}
Let $0\le \psi \le 1$ be a cut-off function such that
\[
\psi = 1 \quad\text{in } D_{2\rho},
\qquad
\psi = 0 \quad\text{in } D_{3\rho}^c.
\]
We decompose $\calF_2(T_h v) - \calF_2(v) = \sum_{i=1}^4 F_i$ as follows:
\begin{align*}
  F_1 & := \int_{\Rd}\int_{\Rd} \psi(x)\psi(y)
  \frac{\big(T_hv(x) - T_hv(y) \big)^2-\big( v(x) - v(y) \big)^2}{|x-y|^{d+2s}} dy dx,
  \\
  F_2 & := \int_{\Rd}\int_{\Rd} \big(1-\psi(x)\big) \psi(y)
  \frac{\big(T_hv(x) - T_hv(y) \big)^2-\big( v(x) - v(y) \big)^2}{|x-y|^{d+2s}} dy dx,
  \\
  F_3 & := \int_{\Rd}\int_{\Rd} \psi(x) \big( 1 - \psi(y)\big)
  \frac{\big(T_hv(x) - T_hv(y) \big)^2-\big( v(x) - v(y) \big)^2}{|x-y|^{d+2s}} dy dx,
  \\
  F_4 & := \int_{\Rd}\int_{\Rd} \big(1 - \psi(x) \big) \big( 1 - \psi(y) \big)
  \frac{\big(T_hv(x) - T_hv(y) \big)^2-\big( v(x) - v(y) \big)^2}{|x-y|^{d+2s}} dy dx,
\end{align*}
We observe that \eqref{Thv=v} implies
\[
T_h v(x) - T_hv(y) = v(x) - v(y) \quad\forall \,  x,y \notin D_{2\rho},
\]
whence $F_4 = 0$. We also realize that $F_3=F_2$ upon exchanging the variables $x$ and $y$ in $F_3$.
We next estimate the remaining two terms $F_2$ and $F_1$.

\medskip
{\it Step 3: Estimate of $F_2$.} We resort to \eqref{z1-z2} to split $F_2 = F_{21} + F_{22}$ with
\begin{align*}
F_{21} &= \int_{\Rd} \int_{\Rd} \big(1-\psi(x)\big) \psi(y)
\frac{\big(T_hv(x) - T_hv(y) \big)^2- z(x,y)^2}{|x-y|^{d+2s}} dy dx,
\\
F_{22} &= \int_{\Rd} \int_{\Rd} \big(1-\psi(x)\big) \psi(y)
  \frac{z(x,y)^2 -\big( v(x) - v(y) \big)^2}{|x-y|^{d+2s}} dy dx,
\end{align*}
and observe that
\begin{align*}
  \big( T_hv(x) - T_hv(y) \big)^2 - z(x,y)^2 &=
  \big(\phi(x) - \phi(y)  \big)^2 \big( v_h(y) - v(y) \big)^2
  \\
  & + 2 \big(\phi(x) - \phi(y)  \big) \big( v_h(y) - v(y) \big) z(x,y).
\end{align*}
This yields $F_{21} = F_{211} + F_{212}$ with
\begin{align*}
F_{211} &= \int_{\Rd} \int_{\Rd} \big(1-\psi(x)\big) \psi(y)
\frac{\big(\phi(x) - \phi(y)  \big)^2 \big( v_h(y) - v(y) \big)^2}{|x-y|^{d+2s}} dy dx,
\\
F_{212} &= 2 \int_{\Rd} \int_{\Rd} \big(1-\psi(x)\big) \psi(y)
\frac{\big(\phi(x) - \phi(y)  \big) \big( v_h(y) - v(y) \big) z(x,y)}{|x.-y|^{d+2s}} dy dx.
\end{align*}
To estimate $F_{211}$ we reorder the integral
\[
F_{211} = \int_{\Rd}  \psi(y) \big( v_h(y) - v(y) \big)^2
\bigg(\int_{\Rd} \big(1-\psi(x)\big)\frac{\big(\phi(x) - \phi(y)  \big)^2}{|x-y|^{d+2s}} dx\bigg) dy
\]
and note that the inner integral is bounded by
\begin{align*}
  \int_{\Rd} \frac{\big(\phi(x) - \phi(y)  \big)^2}{|x-y|^{d+2s}} dx &=
  \int_{D_1(y)} \frac{\big(\phi(x) - \phi(y)  \big)^2}{|x-y|^{d+2s}} dx +
  \int_{D_1(y)^c} \frac{\big(\phi(x) - \phi(y)  \big)^2}{|x-y|^{d+2s}} dx
  \\
  & \le |\phi|_{W^1_\infty(\Rd)}^2 \int_{D_1(y)} \frac{dx}{|x-y|^{d+2s-2}} +
  \int_{D_1(y)^c} \frac{dx}{|x-y|^{d+2s}}
  \lesssim \|\phi\|_{W^1_\infty(\Rd)}^2
\end{align*}
for all $y\in D_{3\rho}$. Therefore
\[
F_{211} \lesssim \|\phi\|_{W^1_\infty(\Rd)}^2 \int_{D_{3\rho}} \big( v_h(y) - v(y) \big)^2
\lesssim |h|^{2\sigma} \|\phi\|_{W^1_\infty(\Rd)}^2 |v|_{B^\sigma_{2,q} (D_{4\rho})}^2.
\]
For the other term, we employ the Cauchy-Schwarz inequality to get
\[
F_{212} \le 2 \sqrt{F_{211}} \bigg(\int_{\Rd} \int_{\Rd} (1-\psi(x)\big) \psi(y)
\frac{z(x,y)^2}{|x-y|^{d+2s}} dy dx \bigg)^{\frac12}.
\]
Using the convexity of $t\mapsto t^2$ we infer that
\begin{equation}\label{convexity}
  z(x,y)^2 \le \phi(x) \big( v_h(x) - v_h(y) \big)^2
  + \big( 1 - \phi(x)\big) \big( v(x) - v(y) \big)^2
\end{equation}
and
\begin{align*}
  \int_{\Rd} \int_{\Rd} \big( 1 - \psi(x) \big) \psi(y) \frac{z(x,y)^2}{|x-y|^{d+2s}} dy dx &\le
  \int_{\Rd} \int_{\Rd} \big( 1 - \psi(x) \big) \phi(x) \psi(y) 
  \frac{ \big( v_h(x) - v_h(y) \big)^2}{|x-y|^{d+2s}} dy dx
  \\
  & + \int_{\Rd} \int_{\Rd} \big( 1 - \psi(x) \big) \big( 1 - \phi(x) \big) \psi(y) 
  \frac{ \big( v(x) - v(y) \big)^2 }{|x-y|^{d+2s}} dy dx.
\end{align*}
We see that this the first term vanishes because $\big( 1 - \psi(x) \big) \phi(x)=0$ for all
$x\in\Rd$. For the second term we use that 
$\big(1 - \psi(x) \big) \big( 1 - \phi(x) \big) \le 1$ for all $x \in\Rd$ and
$\supp \psi \subset D_{3\rho}$ to obtain
\begin{equation*} \begin{split}
  \int_{\Rd} \int_{\Rd} \big( 1 - \psi(x) \big) \big( 1 - \phi(x) \big) \psi(y) \ &
  \frac{ \big( v(x) - v(y) \big)^2 }{|x-y|^{d+2s}} dy dx \\
&  \le
  \int_{D_{3\rho}}\int_{\Rd} \frac{\big( v(x) - v(y) \big)^2}{|x-y|^{d+2s}} dx dy  = |v|_{H^s(D_{3\rho},\Rd)}^2.
\end{split}\end{equation*}
Collecting the previous estimates, we arrive at
\begin{equation}\label{F21}
F_{21} \lesssim |h|^\sigma \|\phi\|_{W^1_\infty(\Rd)} |v|_{B^\sigma_{2,q}(D_{4\rho})} |v|_{H^s(D_{3\rho},\Rd)}.
\end{equation}

We now turn to term $F_{22}$. In light of \eqref{convexity}, we see that
\[
z(x,y)^2 - \big( v(x) - v(y) \big)^2
\le \phi(x)  \Big( \big( v_h(x) - v_h(y) \big)^2 - \big( v(x) - v(y) \big)^2 \Big),
\]
whence
\[
F_{22} \le \int_{\Rd} \int_{\Rd} \big( 1 - \psi(x) \big) \phi(x) \psi(y)
\frac{\big( v_h(x) - v_h(y) \big)^2 - \big( v(x) - v(y) \big)^2}{|x-y|^{d+2s}} dy dx.
\]
We point out that the localization occurs because of the presence of $\phi(x)$ and $\psi(y)$
in the integrand. Upon changing variables to convert $v_h$ into $v$ we obtain
the equivalent expression of $F_{22}$
\[
F_{22} = \int_{\Rd} \int_{\Rd} \underbrace{\Big(\big(1 - \psi_{-h}(x) \big) \phi_{-h}(x) \psi_{-h}(y)
- \big(1 - \psi(x) \big) \phi(x) \psi(y) \Big)}_{=K(x,y)} \frac{\big( v(x) - v(y) \big)^2}{|x-y|^{d+2s}} dy dx.
\]
We now observe that the kernel $K(x,y)$ is local and
\begin{align*}
  \big| K(x,y) \big| & \le
  \big| \psi_{-h}(x) - \psi(x) \big| \phi_{-h}(x) \psi_{-h} (x)
  \\
  & + \big( 1-\psi(x) \big) \big| \phi_{-h}(x) - \phi(x) \big| \psi_{-h}(y)
  \\
  & + \big( 1-\psi(x) \big) \big| \phi(x) \big| \psi_{-h}(y) - \psi(y) \big|
  \\
  & \lesssim |h|\big( \|\psi\|_{W^1_\infty(\Rd)} +  \|\phi\|_{W^1_\infty(\Rd)} \big)
  \chi_{D_{4\rho}} (x) \chi_{D_{4\rho}}(y).
\end{align*}
Consequently,
\begin{equation}\label{F22}
F_{22} \lesssim |h| \big( \|\psi\|_{W^1_\infty(\Rd)} +  \|\phi\|_{W^1_\infty(\Rd)} \big) |v|_{H^s(D_{4\rho})}^2.
\end{equation}
Since $|v|_{H^s(D_{4\rho})} \le |v|_{B^\sigma_{2,q} (D_{4\rho})}$ under the assumptions on $\sigma$ and $q$,
\eqref{F21} and \eqref{F22} together yield
\begin{align}\label{F2}
  F_2 \lesssim \Big( |h|^\sigma |v|_{B^\sigma_{2,q} (D_{4\rho})} + |h| |v|_{H^s (D_{4\rho})}  \Big)
  |v|_{H^s(D_{4\rho},\Rd)}
  \lesssim |h|^\sigma |v|_{B^\sigma_{2,q} (D_{4\rho})} |v|_{H^s(D_{4\rho,}\Rd)},
\end{align}
where we have not written the dependence on $\|\phi\|_{W^1_\infty(\Rd)}$ and $\|\psi\|_{W^1_\infty(\Rd)}$
to simplify the notation. 

\medskip
{\it Step 4: Estimate of $F_1$.}
This term is already localized because of the factor $\psi(x) \psi(y)$. In any event, a similar
argument to Step 3 yields $F_1 = F_{11} + F_{12}$ with
\begin{align*}
  F_{11} &= \int_{\Rd} \int_{\Rd} \psi(x)\psi(y)
  \frac{\big( T_hv(x) - T_hv(y)\big)^2 - z_1(x,y)^2}{|x-y|^{d+2s}} dy dx
  \\
  F_{12} &= \int_{\Rd} \int_{\Rd} \psi(x)\psi(y)
  \frac{z_1(x,y)^2 - \big( v(x) - v(y)\big)^2}{|x-y|^{d+2s}} dy dx.
\end{align*}
Repeating Step 3 with $1-\psi(x)$ replaced by $\psi(x)$ we arrive at the following local estimate:
\begin{equation}\label{F1}
F_1 \lesssim |h|^\sigma |v|_{B^\sigma_{2,q}(D_{4\rho})} |v|_{H^s(D_{4\rho})}.
\end{equation}

\medskip
{\it Step 5: Final estimate.}
Combining \eqref{F1} with \eqref{F2} and recalling that $F_4=0$ and $F_2=F_3$,
yields the final estimate
\[
\big| \calF_2(T_hv) - \calF_2(v)  \big| \lesssim
|h|^\sigma |v|_{B^\sigma_{2,q}(D_{4\rho})} |v|_{H^s(D_{4\rho},\Rd)},
\]
which is the desired localized estimate \eqref{eq:higher-regularity-F0}.
\end{proof}

\begin{remark}[maximal regularity gain]
Heuristically, proving $\sigma$-regularity of the functionals translates into a $\sigma/2$ regularity pickup in the solutions of our problem. We emphasize that the proof of Proposition \ref{prop:regularity-F0} does not yield $\sigma$-regularity of the functional $\calF_2$ for any $\sigma > 1$.  This indicates that we should not expect solutions to \eqref{eq:Dirichlet} to pick up more than $1/2$ derivative, independently of the smoothness of $f$. This is consistent with \eqref{eq:example}, where the problem is posed in a ball and $f$ is constant, but \rhn{$u \notin \widetilde{H}^{s+1/2}(\Omega)$.} Evidently, if $f$ possesses limited regularity then the functional $\calF_1$ may become the bottleneck in the regularity of $\calF$ (cf. \eqref{eq:regularity-f-Besov}).
\end{remark}

\section{Besov regularity for $L^2$ data: proof of Theorem \ref{thm:regularity-solutions}} \label{sec:Besov-regularity}

We now establish Besov regularity of solutions to \eqref{eq:Dirichlet} in case $f \in L^2(\Omega)$ and
$\Omega$ is Lipschitz, namely, we give a proof of Theorem \ref{thm:regularity-solutions}. 
\rhn{We first point out that the desired estimate \eqref{eq:regularity-u-low} for $s\in (1/2,1)$ follows from Theorem \ref{prop:regularity-solutions-rough} (Besov regularity for rough data) upon using estimate \eqref{eq:embedding-negative} with $r = s - 1/2 \in (0,1/2)$. Therefore, we now focus on the following reduced form of Theorem \ref{thm:regularity-solutions}.

\begin{theorem}[Besov regularity with $L^2$-data and {$s\in (0,1/2]$}] \label{thm:regularity-solutions-rep}
Let $\Omega$ be a bounded Lipschitz domain and $f \in L^2(\Omega)$. 
If $s \in (0,1/2)$, then the solution $u$ to \eqref{eq:Dirichlet} satisfies $u \in \dot{B}^{2s}_{2,\infty}(\Omega)$ and
\begin{equation}\label{eq:regularity-u-low-again}
\| u\|_{\dot{B}^{2s}_{2,\infty}(\Omega)} \le \frac{C(\Omega,d)}{\sqrt{1-2s}} \| f \|_{L^2(\Omega)} .
\end{equation}
On the other hand, the solution for $s = 1/2$ satisfies $u \in \dot{B}^{1-\epsilon}_{2,\infty}(\Omega)$ \jp{for every $0<\epsilon < 1$ and}
\begin{equation} \label{eq:regularity-u-s1/2-again}
\| u \|_{\dot{B}^{1-\epsilon}_{2,\infty}(\Omega)} \le \frac{C(\Omega,d)}{\sqrt{\epsilon}} \| f \|_{L^2(\Omega)}.
\end{equation}
\end{theorem}
}

\begin{proof}
Since $f \in L^2(\Omega) \subset H^{-s}(\Omega)$, problem \eqref{eq:Dirichlet} is well-posed; according to \eqref{eq:well-posedness}, it holds that
\begin{equation} \label{eq:stability-Hs}
|u|_{H^s(\Rd)} \le \|f\|_{H^{-s}(\Omega)} \le C \|f\|_{L^{2}(\Omega)},
\end{equation}
where $C$ depends on $\Omega$.
Because $\Omega$ is a Lipschitz domain, by Proposition \ref{prop:cone-property} (uniform cone property) there exist $\rho$, $\theta$ such that the cone $C_\rho(x)=\calC_\rho (\vn(x), \theta)$ of height $\rho$, opening $\theta$ and axis $\vn(x)$ is made of admissible directions $h$ for all $x \in \Rd$, namely $C_\rho(x) \subset \calO_\rho(x)$; see Definition \ref{D:admissible-vectors}.

We consider a finite covering of the domain $\Omega^\rho=\{x\in\Rd: \mbox{dist}(x,\Omega)<\rho\}$ by balls $D_\rho(x_j) = D(x_j, \rho)$ of radius $\rho$ and center $x_j\in\Omega^\rho$, and fix the cone $C_j = C_\rho (x_j)$ of admissible directions, for each $j = 1, \ldots , M$. By the localization estimate \eqref{eq:localization2}, it suffices to bound Besov seminorms over each of the balls $D_j$ using  $C_j$. We split the argument into several steps. 

\medskip
{\it Step 1: Regularity improvement.}
We consider one of the balls $D_\rho(x_j)$ in the covering and corresponding cone $C_j = \calC_\rho (\vn(x_j), \theta)$. Importantly, if $h \in C_j$ we can guarantee that $T_h u \in \tHs$ and because $T_h u = u_h$ in $D_\rho(x_j)$, we combine Proposition \ref{prop:bound-Besov} (bounds on Besov seminorms) and Lemma \ref{lem:corollary1} (regularity and minimizers) to deduce that, for $\sigma \in [s,1]$,
\[
| u |^2_{B^{s+\sigma/2}_{2,\infty}(D_\rho(x_j))} \lesssim \sup_{h \in C_j} 
\frac{| u - T_hu|^2_{H^s(D_\rho(x_j))}}{|h|^{\sigma}} \lesssim \omega(u; \calF, T, C_j, \sigma).
\]
  
To bound the $(T,C_j,\sigma)$-regularity modulus $\omega$ of $\calF=\calF_2-\calF_1$, we \red{now assume $\sigma \in [s,1)$,} exploit the subadditivity \eqref{eq:subadditivity} of $\omega$, and \red{combine either \eqref{eq:regularity-f} if $\sigma > s$ or \eqref{eq:regularity-f-Besov} (with $q=2$ and $\gamma = 0$) if $\sigma = s$ with \eqref{eq:higher-regularity-F0} to obtain:}
\begin{equation}\label{eq:local-reg}
\begin{aligned}
\omega(u; \calF, T, C_j, \sigma) & \le \omega(u; \calF_1, T, C_j, \sigma) + \omega(u; \calF_2, T, C_j, \sigma) \\
& \lesssim  \|u\|_{B^\sigma_{2,q}(D_{4\rho}(x_j))} \left( \red{\frac{\| f \|_{L^2(D_{2\rho}(x_j) \cap \Omega)}}{\sqrt{1-\sigma}} } +  |u|_{H^s(D_{4\rho}(x_j),\Rd)} \right).
\end{aligned}
\end{equation}
\red{\rhn{Note that $q$ in \eqref{eq:local-reg} is} $q=2$ if $\sigma = s$ and $q=\infty$ if $\sigma > s$.}
Using the localization estimate \eqref{eq:localization2}, and applying Cauchy-Schwarz in conjunction
with stability estimate \eqref{eq:stability-Hs}, we end up with
\begin{equation} \label{eq:bound3s/2}
\begin {split}
\| u \|^2_{\dot{B}^{s+\sigma/2}_{2,\infty}(\Omega)} & \lesssim  \sum_{j=1}^M \| u \|^2_{B^{s+\sigma/2}_{2,\infty}(D_\rho(x_j))}
\\ &\lesssim \|u\|_{\dot{B}^\sigma_{2,q}(\Omega)} \left(\red{\frac{\| f \|_{L^2(\Omega)}}{\sqrt{1-\sigma}} }  + |u|_{H^s(\Rd)} \right) \lesssim 
\red{\frac{1}{\sqrt{1-\sigma}} }  \|u\|_{\dot{B}^\sigma_{2,q}(\Omega)} \| f \|_{L^2(\Omega)},
\end{split}\end{equation}
\rhn{with a hidden constant $\Lambda=\Lambda(M)$ depending on the cardinality $M$ of the covering. Therefore, increasing $\Lambda$ further to make it as large as the constant $C$ in \eqref{eq:stability-Hs}, \eqref{eq:bound3s/2} reads equivalently}
\begin{equation}\label{eq:Lambda}
\red{ \| u \|^2_{\dot{B}^{s+\sigma/2}_{2,\infty}(\Omega)} \le \frac{\Lambda}{\sqrt{1-\sigma}}  \|u\|_{\dot{B}^\sigma_{2,q}(\Omega)} \| f \|_{L^2(\Omega)}.}
\end{equation}
Making the stability constant $\Lambda$ explicit is crucial for the bootstrap argument below.

To utilize this estimate, we must know that $u \in \dot{B}^\sigma_{2,q}(\Omega)$ for some \red{$\sigma \in [s,1)$,} $q \in [1,\infty]$ (with $q \le 2$ if $\sigma = s$). At the beginning of our argument, we know this to hold only for $\sigma = s$, $q = 2$ with the \rhn{stability bound \eqref{eq:stability-Hs}.} This implies the first improved estimate with $C$ depending on $\Omega$ \red{and $s$}
\[
\| u \|^2_{\dot{B}^{3s/2}_{2,\infty}(\Omega)}  \le \Lambda  \| f \|_{H^{-s}(\Omega)} \| f \|_{L^2(\Omega)}
\le C \Lambda \| f \|_{L^2(\Omega)}^2,
\]
whence $u \in \dot{B}^{3s/2}_{2,\infty}(\Omega)$. We shall iterate \eqref{eq:Lambda} to improve further the regularity of $u$.

\medskip
\red{{\it Step 2: Regularity for $s \in (0, 1/2)$. } 
We let $s \in (0,1/2)$ and consider the sequence
\[
\rhn{\sigma_j = 2s \left( 1 - \frac1{2^j} \right), \quad j \ge 0,} 
\]
which is monotone increasing and satisfies $\sigma_j \to 2s^-$ and the recursion relation
\[
s + \frac{\sigma_j}{2} = \sigma_{j+1} \quad \rhn{\forall j \ge 0.}
\]

To prove that $u \in \dot{B}^{2s}_{2,\infty}(\Omega)$, we claim that
\begin{equation}\label{eq:regularity-induction}
  \rhn{u \in \dot{B}^{\sigma_j}_{2,q}(\Omega),} \quad
  \|u\|_{\dot{B}^{\sigma_j}_{2,q}(\Omega)} \le \rhn{\frac{\Lambda}{\sqrt{1-\sigma_{j-1}}}} \| f \|_{L^2(\Omega)}  \quad \forall j \ge 1 , 
\end{equation}
\rhn{where $q = 2$ if $j = 1$ and $q = \infty$ otherwise, and $\Lambda$ is the constant in \eqref{eq:Lambda}.} We argue by induction. The claim is true for $j = 1$ in view of \eqref{eq:stability-Hs}, because $\sigma_0 = s$, \rhn{ $q=2$ and $\Lambda\ge C$ the constant in \eqref{eq:stability-Hs}.}
Let $j \ge 1$ and assume \eqref{eq:regularity-induction} holds. We apply \eqref{eq:Lambda} and exploit the fact that \rhn{$1-\sigma_j < 1 - \sigma_{j-1}$} to get
\[
\|u\|_{\dot{B}^{\sigma_{j+1}}_{2,\infty}(\Omega)} \le \rhn{\frac{\Lambda^{1/2}}{(1-\sigma_j)^{1/4}} \|u\|_{\dot{B}^{\sigma_j }_{2,q}(\Omega)}^{1/2}} \| f \|_{L^2(\Omega)}^{1/2} \le \rhn{\frac{\Lambda}{\sqrt{1-\sigma_j}}} \| f \|_{L^2(\Omega)} ,
\]
which shows the validity of \eqref{eq:regularity-induction} for $j+1$.
This implies that $u \in \dot{B}^{2s-\epsilon}_{2,\infty}(\Omega)$ for all $\epsilon > 0$ and
\[
\|u\|_{\dot{B}^{2s-\epsilon}_{2,\infty}(\Omega)} \le \rhn{\frac{\Lambda}{\sqrt{1-2s}}} \| f \|_{L^2(\Omega)}.
\]
It remains to prove \eqref{eq:regularity-u-low-again}.
To this end, we resort to Lemma \ref{lemma:Besov} \rhn{(equivalence of Besov seminorms)} which
expresses the Besov seminorms in terms of \rhn{second-order} difference quotients without weights:
\[
\sup_{|h|\le\rho} \int_{\Rd} \frac{\rhn{\left| \delta_2(h) u(x) \right|^2}}{|h|^{4s - 2\epsilon}} dx  \simeq |u|_{\dot{B}^{2s-\epsilon}_{2,\infty}(\Omega)}^2 \le \rhn{\frac{\Lambda^2}{1-2s}} \| f \|_{L^2(\Omega)}^2.
\]
Consider the pointwise non-decreasing sequence of functions
\[
w_\epsilon(x) := \frac{\rhn{\left| \delta_2(h) u(x) \right|^2}}{|h|^{4s - 2\epsilon}}
\quad\uparrow\quad w_0(x) := \frac{\rhn{\left| \delta_2(h) u(x) \right|^2}}{|h|^{4s}}
\quad\textrm{as } \epsilon\to0
\]
for every $|h|\le\rho\le1$, whence the Monotone Convergence Theorem yields
\[
\int_{\Rd} w_0(x) dx  = \lim_{\epsilon\to 0} \int_{\Rd} w_\epsilon(x) dx \le
\rhn{\frac{\Lambda^2}{1-2s}} \| f \|_{L^2(\Omega)}^2 .
\]
This is the desired estimate \eqref{eq:regularity-u-low-again} in disguise, namely
$\|u\|_{\dot{B}^{2s}_{2,\infty}(\Omega)} \le \rhn{\frac{\Lambda}{\sqrt{1-2s}}} \| f \|_{L^2(\Omega)}$.
}

\red{
\medskip
{\it Step 3: Regularity for $s = 1/2$.}
We take the same sequence $\{\sigma_j\}$ as in \rhn{Step 2, which now reads $\sigma_j = 1 - \frac1{2^j}$. The expression \eqref{eq:regularity-induction} for $q=\infty$ becomes} 
\[
\| u \|_{\dot{B}^{1 - 2^{-j}}_{2,\infty}(\Omega)} \le \rhn{\Lambda 2^{(j-1)/2}}  \| f \|_{L^2(\Omega)} \quad \forall j \ge 1.
\]
Given any $\epsilon \in (0,1)$, we let $j$ be the integer number such that $\epsilon \in [2^{-j},2^{-(j-1)})$ and observe that \rhn{$2^{(j-1)/2} \le \epsilon^{-1/2}$.}
Thus, we deduce
\[
\| u \|_{\dot{B}^{1 - \epsilon}_{2,\infty}(\Omega)} \le \rhn{\frac{\Lambda}{\sqrt{\epsilon}}} \| f \|_{L^2(\Omega)}
\]
and this finishes the proof.
}
\end{proof}

\begin{remark}[Sobolev regularity]
\rhn{We now invoke the embedding \eqref{eq:embedding} between Besov and Sobolev spaces to derive more familiar estimates. If $s \neq 1/2$ and $r = \min\{2s,s+1/2\}$, then \eqref{eq:regularity-u-low} implies}
\begin{equation}\label{eq:Sobolev}
|u|_{H^{r-\epsilon}(\Rd)} \le \red{\frac{C}{\sqrt{ \epsilon \, |1-2s|}}}\|f\|_{L^2(\Omega)}
\end{equation}
for any \rhn{$\epsilon < \max\{r,1/4\}$. Similarly, if $s=1/2$ then \eqref{eq:regularity-u-s1/2} yields}
\red{
\begin{equation} \label{eq:Sobolev-s1/2}
|u|_{H^{1-\epsilon}(\Rd)} \le \frac{C}{\epsilon} \|f\|_{L^2(\Omega)}.
\end{equation}
Estimates of this type have been derived by Grubb \cite{Grubb15} for bounded domains with
$C^\infty$ boundary. In contrast to the integer case $s=1$, \eqref{eq:Sobolev} and \eqref{eq:Sobolev-s1/2} confirm} that
the presence of reentrant corners does not reduce the regularity of $u$ for $f\in L^2(\Omega)$.
\end{remark}

\begin{remark}[gap for smooth data]
Grubb \cite{Grubb15} showed that the solution $u$ of \eqref{eq:Dirichlet} belongs to
$H^{s+1/2-\epsilon}(\Omega) \setminus H^{s+1/2}(\Omega)$
for any $\epsilon>0$ even for data $f$ smoother than $L^2(\Omega)$ on $C^\infty$
domains $\Omega$. \rhn{More recently, Abels and Grubb \cite{AbelsGrubb} reduced the domain regularity to $C^{1,\beta}$
with $\beta>2s$. This is consistent with \eqref{eq:Sobolev} for $s>1/2$. The proof of Theorem \ref{thm:regularity-solutions-rep} reveals a different type of regularity obstruction for $s\le 1/2$.} Since our bootstrapping argument pivots on $\widetilde{H}^s(\Omega)$, a limitation emerges from the asymptotics of the regularity parameter $\sigma_j\ge s$, which satisfies the recursion
\[
\sigma_{j+1} = s + \frac{\sigma_j}{2}
\quad\Rightarrow\quad
\sigma_j\to 2s^+
\]
as $j\to\infty$. Therefore, for $s \le 1/2$ the regularity enhancement is $s$ rather than $1/2$ even
for data $f$ smoother than $L^2(\Omega)$. \rhn{This gap between $s$ and $1/2$
remains out of reach for our technique, but it has been bridged by a novel argument in \cite{BoLiNo:Barrett-22,BoLiNo22} that leads to the optimal shift property \eqref{eq:optimal-shift}.}
\end{remark}

\section{Besov regularity for rough data: Proofs of Theorem \ref{prop:regularity-solutions-rough} and Corollary \ref{cor:intermediate}} \label{sec:less-regular-f}

We now consider rough data, namely data in negative-order Besov spaces.

\begin{proof}[Proof of Theorem \ref{prop:regularity-solutions-rough}] We prove the ideal shift inequality \eqref{eq:s+1/2} for $s\in(1/2,1)$, namely
\begin{equation}\label{eq:max-regularity-s-ge-12}
\|u\|_{\dot{B}^{s+1/2}_{2,\infty}(\Omega)} \lesssim \|f\|_{B^{-s+1/2}_{2,1}(\Omega)},
\end{equation}
with conjugate indices $q=\infty$ and $q=1$.
To this end, we perform a bootstrapping argument similar to the proof of
Theorem \ref{thm:regularity-solutions-rep} \rhn{(Besov regularity with $L^2$-data and $s\in(0,1/2]$); the main difference is that we can now take $\sigma = 1$ in the bound \eqref{eq:regularity-f-Besov} of $\calF_1$, the most delicate functional.}

\medskip
\rhn{\it Step 1: Regularity improvement.}
We consider a covering $\{D_\rho(x_j)\}$ of $\Omega$ with balls like in Step 1 of that proof, use the subadditivity \eqref{eq:subadditivity} of $\calF = \calF_2-\calF_1$ together with \eqref{eq:regularity-f-Besov}, \eqref{eq:higher-regularity-F0} and \eqref{eq:well-posedness} to obtain
\begin{align*}
\omega(u; \calF, T, C_j, \sigma)
\lesssim  \| u \|_{B^{\sigma+\gamma}_{2,q}(D_{3\rho}(x_j))} \| f \|_{B^{-\gamma}_{2,q'}(\Omega)}
+ \|u\|_{B^\sigma_{2,q}(D_{4\rho}(x_j))} \|f\|_{H^{-s}(\Omega)},
\end{align*}
on an arbitrary ball $D_\rho(x_j)$ with \rhn{$\gamma \in (0,s)$ and $\sigma \in [0,1]$}. Since
$\|u\|_{B^{s+\sigma/2}_{2,\infty}(D_\rho(x_j))}^2 \lesssim \omega(u; \calF, T, C_j, \sigma)$ and
\rhn{$f \in B^{-\gamma}_{2,q'}(\Omega) \subset H^{-s}(\Omega)$,} we deduce
\begin{equation*}
  \|u\|_{B^{s+\sigma/2}_{2,\infty}(D_\rho(x_j))}^2 \lesssim
  \| u \|_{B^{\sigma+\gamma}_{2,q}(D_{4\rho}(x_j))} \| f \|_{B^{-\gamma}_{2,q'}(\Omega)},
\end{equation*}    
We observe that, in contrast to \eqref{eq:local-reg}, we do not localize
either $\| f \|_{B^{-\gamma}_{2,q'}(\Omega)}$ or $|u|_{\widetilde{H}^s(\Omega)}$, the latter giving rise to
$\|f\|_{H^{-s}(\Omega)}$. This is because doing so would require dealing with localization of
positive Besov norms, as stated in Lemma \ref{L:localization} (localization), but the equivalence constants
are sensitive to the \red{cardinality} $M$ of the covering. Therefore, instead of \eqref{eq:bound3s/2},
adding over $1\le j \le M$ \rhn{and using that $B^{-\gamma}_{2,1}(\Omega) \subset B^{-\gamma}_{2,q'}(\Omega)$}
we obtain the bound
\begin{equation} \label{eq:bound-sigma-gamma} \begin{aligned}
\| u \|^2_{\dot{B}^{s+\sigma/2}_{2,\infty}(\Omega)}
\le \Lambda \| u \|_{\dot{B}^{\sigma+\gamma}_{2,q}(\Omega)} \rhn{\| f \|_{B^{-\gamma}_{2,1}(\Omega)}}
\end{aligned} \end{equation}
with $\Lambda>0$ depending on $M$. \rhn{We further assume that $\Lambda$ is as large as the constant $C$ in the estimate $\|f\|_{H^{-s}(\Omega)} \le C \|f\|_{B^{-\gamma}_{2,1}(\Omega)}$.} The maximal regularity we may expect corresponds to $\sigma = 1$
\begin{equation*} \label{eq:max-regularity-gamma} 
\| u \|^2_{\dot{B}^{s+1/2}_{2,\infty}(\Omega)} \le \Lambda
\rhn{\| u \|_{\dot{B}^{1+\gamma}_{2,q}(\Omega)} \| f \|_{B^{-\gamma}_{2,1}(\Omega)},}
\end{equation*}
which coincides with \eqref{eq:max-regularity-s-ge-12} with $q=\infty$ \rhn{provided $\gamma := s-1/2 > 0$.}

\medskip
\rhn{{\it Step 2: Bootstrap argument.}
To exploit \eqref{eq:bound-sigma-gamma}, we define the sequence $\{\sigma_j\}$ by recursion, namely}
\[
s + \frac{\sigma_j}{2} = \sigma_{j+1} + \gamma = \sigma_{j+1} + s -\frac12
\quad\Rightarrow\quad
\sigma_{j+1} = \frac{\sigma_j+1}{2},
\]
with initial value $\sigma_0=1/2$; the latter is a consequence of setting the starting value
\rhn{$\sigma_0+\gamma=s$ in \eqref{eq:bound-sigma-gamma} with $\gamma=s-1/2$ and $q=2$.}
Using an induction argument, we readily see that
\begin{itemize}
\item
  $\frac12 \le \sigma_j \le 1$: $\sigma_{j+1} = \frac{\sigma_j}{2} + \frac12 \le 1$;

\item
  $\sigma_j$ is monotone increasing: $\sigma_{j+1} = \frac{\sigma_j}{2} + \frac12
  \ge \frac{\sigma_j}{2} + \frac{\sigma_j}{2} = \sigma_j$.
\end{itemize}
This implies that $\sigma_j$ converges and $\lim_{j\to\infty} \sigma_j = 1$.
\rhn{We claim that $u\in \dot{B}^{s+\sigma_j/2}_{2,\infty}(\Omega)$ and
\begin{equation}\label{eq:reg-induction-2}
  \| u \|_{\dot{B}^{s+\sigma_j/2}_{2,\infty}(\Omega)} \le \Lambda \| f \|_{B^{-s+1/2}_{2,1}(\Omega)}
  \quad\forall \, j\ge 0,
\end{equation}
provided $\Lambda$ is as in \eqref{eq:bound-sigma-gamma}. We prove \eqref{eq:reg-induction-2} by induction. For $j=0$, \eqref{eq:bound-sigma-gamma} with $q=2$ and \eqref{eq:well-posedness} yield
\[
\| u \|_{\dot{B}^{s+\sigma_0/2}_{2,\infty}(\Omega)}^2 \le \Lambda \| u \|_{\widetilde{H}^{s}(\Omega)}  
\| f \|_{B^{-s+1/2}_{2,1}(\Omega)}
\le \Lambda \| f \|_{H^{-s}(\Omega)} \| f \|_{B^{-s+1/2}_{2,1}(\Omega)} \le \Lambda^2 \| f \|_{B^{-s+1/2}_{2,1}(\Omega)}^2,
\]
Moreover, if we assume \eqref{eq:reg-induction-2} valid for $j-1\ge0$, then \eqref{eq:bound-sigma-gamma} gives
\eqref{eq:reg-induction-2} for $j$
\[
\| u \|_{\dot{B}^{s+\sigma_j/2}_{2,\infty}(\Omega)}^2 \le \Lambda \| u \|_{\dot{B}^{s+\sigma_{j-1}/2}_{2,\infty}(\Omega)}  
\| f \|_{B^{-s+1/2}_{2,1}(\Omega)}
\le \Lambda^2  \| f \|_{B^{-s+1/2}_{2,1}(\Omega)}^2,
\]
}
as well as \eqref{eq:max-regularity-s-ge-12} upon letting $j\to\infty$ and arguing
as in \red{Step 2} of the proof of Theorem \ref{thm:regularity-solutions-rep}.
\end{proof}

\begin{proof}[Proof of Corollary \ref{cor:intermediate}] \rhn{We resort to the interpolation properties \eqref{eq:def-negative-Besov} and \eqref{eq:interpolation_Besov}.} \red{We consider first the case $s \in (0,1/2)$ and interpolate between the regularity bound \eqref{eq:regularity-u-low-again}} and the stability estimate \eqref{eq:well-posedness}, to immediately deduce that \eqref{eq:s+theta} holds for $\theta \in (0, s), q \in [1,\infty]$, namely,
\[
\|u\|_{\dot{B}^{s+\theta}_{2,q}(\Omega)}  
  \jp{\le \frac{C}{(1-2s)^{\theta/2s}}} \| f \|_{B^{-s+\theta}_{2,q}(\Omega)},
\]
\jp{where $C$ depends on $\Omega,d$ and $q$.}
\red{Instead, if $s = 1/2$ we interpolate between \eqref{eq:regularity-u-s1/2-again} and \eqref{eq:well-posedness} in a similar fashion to derive \eqref{eq:1/2+theta}.}
Finally, if $s \in (1/2,1)$ we interpolate between \eqref{eq:max-regularity-s-ge-12} and  \eqref{eq:well-posedness} to obtain again \red{the estimate \eqref{eq:s+theta},} except that this time $\theta \in (0,1/2)$, $q \in [1,\infty]$.
This concludes the proof of Corollary \ref{cor:intermediate}.
\end{proof}

\subsection*{Acknowledgments}
The authors would like to thank H. Aimar, \red{G. Grubb,} and G. Savar\'e for invaluable discussions about Besov spaces.

\bibliographystyle{abbrv}
\bibliography{regularity.bib}

\begin{thebibliography}{10}

\bibitem{AbatangeloRosOton}
N.~Abatangelo and X.~Ros-Oton.
\newblock Obstacle problems for integro-differential operators: higher
  regularity of free boundaries.
\newblock {\em Adv. Math.}, 360:106931, 2020.

\bibitem{AbelsGrubb}
H.~Abels and G.~Grubb.
\newblock Fractional-order operators on nonsmooth domains.
\newblock {\em arXiv preprint arXiv:2004.10134}, 2020.

\bibitem{AcBo17}
G.~Acosta and J.~P. Borthagaray.
\newblock {A fractional Laplace equation: regularity of solutions and finite
  element approximations}.
\newblock {\em SIAM J. Numer. Anal.}, 55(2):472--495, 2017.

\bibitem{ABBM18}
G.~Acosta, J.~P. Borthagaray, O.~Bruno, and M.~Maas.
\newblock {Regularity theory and high order numerical methods for the
  (1D)-fractional Laplacian}.
\newblock {\em Math. Comp.}, 87(312):1821--1857, 2018.

\bibitem{AdamsFournier:2003}
R.~A. Adams and J.~J.~F. Fournier.
\newblock {\em Sobolev spaces}, volume 140 of {\em Pure and Applied Mathematics
  (Amsterdam)}.
\newblock Elsevier/Academic Press, Amsterdam, second edition, 2003.

\bibitem{BerghLofstrom}
J.~Bergh and J.~L\"ofstr\"om.
\newblock {\em Interpolation spaces: an introduction}.
\newblock Springer-Verlag, Berlin, 1976.

\bibitem{BiWaZu17}
U.~Biccari, M.~Warma, and E.~Zuazua.
\newblock {Local elliptic regularity for the Dirichlet fractional Laplacian}.
\newblock {\em Adv. Nonlinear Stud.}, 17(2):387--409, 2017.

\bibitem{Bogdan:99}
K.~Bogdan.
\newblock Representation of {$\alpha$}-harmonic functions in {L}ipschitz
  domains.
\newblock {\em Hiroshima Math. J.}, 29(2):227--243, 1999.

\bibitem{Bogdan:00}
K.~Bogdan.
\newblock Sharp estimates for the {G}reen function in {L}ipschitz domains.
\newblock {\em J. Math. Anal. Appl.}, 243(2):326--337, 2000.

\bibitem{BoCi18}
J.~P. Borthagaray and P.~Ciarlet~Jr.
\newblock On the convergence in ${H}^1$-norm for the fractional {L}aplacian.
\newblock {\em SIAM J. Numer. Anal.}, 57(4):1723--1743, 2019.

\bibitem{BoLiNo:Barrett-22}
J.~P. Borthagaray, W.~Li, and R.~H. Nochetto.
\newblock Fractional diffusion on lipschitz domains: regularity and
  approximation.
\newblock {\em The 50th John H. Barrett Memorial Lectures: Approximation,
  Applications, and Analysis of Nonlocal, Nonlinear Models, {\rm T. Mengesha
  and A. Salgado eds}}, 2022.

\bibitem{BoLiNo22}
J.~P. Borthagaray, W.~Li, and R.~H. Nochetto.
\newblock {Quasi-linear fractional-order operators in Lipschitz domains}.
\newblock In preparation, 2022.

\bibitem{BdPM}
J.~P. Borthagaray, L.~M.~D. Pezzo, and S.~Mart{\'\i}nez.
\newblock {Finite element approximation for the fractional eigenvalue problem}.
\newblock {\em J. Sci. Comput.}, 77(1):308--329, 2018.

\bibitem{Cozzi17}
M.~Cozzi.
\newblock {Interior regularity of solutions of non-local equations in Sobolev
  and Nikol'skii spaces}.
\newblock {\em Ann. Mat. Pura Appl. (4)}, 196(2):555--578, 2017.

\bibitem{Ditzian88}
Z.~Ditzian.
\newblock On the {M}archaud-type inequality.
\newblock {\em Proc. Amer. Math. Soc.}, 103(1):198--202, 1988.

\bibitem{Dyda2016}
B.~Dyda, A.~Kuznetsov, and M.~Kwa{\'{s}}nicki.
\newblock Fractional {L}aplace operator and {M}eijer {G}-function.
\newblock {\em Constr. Approx.}, 45(3):427--448, 2016.

\bibitem{Eskin}
G.~I. {\`E}skin.
\newblock {\em Boundary value problems for elliptic pseudodifferential
  equations}, volume~52.
\newblock American Mathematical Society, 1981.

\bibitem{Faustmann:20}
M.~Faustmann, M.~Karkulik, and J.~M. Melenk.
\newblock {Local convergence of the FEM for the integral fractional Laplacian}.
\newblock {\em SIAM J. Numer. Anal.}, 60(3):1055--1082, 2022.

\bibitem{Getoor61}
R.~K. Getoor.
\newblock First passage times for symmetric stable processes in space.
\newblock {\em Trans. Amer. Math. Soc.}, 101:75--90, 1961.

\bibitem{Gimperlein:19}
H.~Gimperlein, E.~Stephan, and J.~Stocek.
\newblock Corner singularities for the fractional {L}aplacian and finite
  element approximation.
\newblock Preprint available at
  {\url{https://mat1.uibk.ac.at/heiko/corners.pdf}}, 2019.

\bibitem{Grisvard}
P.~Grisvard.
\newblock {\em Elliptic problems in nonsmooth domains}, volume~24 of {\em
  Monographs and Studies in Mathematics}.
\newblock Pitman (Advanced Publishing Program), Boston, MA, 1985.

\bibitem{Grubb15}
G.~Grubb.
\newblock Fractional {L}aplacians on domains, a development of {H}\"ormander's
  theory of $\mu$-transmission pseudodifferential operators.
\newblock {\em Adv. Math.}, 268:478--528, 2015.

\bibitem{Jakubowski:02}
T.~Jakubowski.
\newblock The estimates for the {G}reen function in {L}ipschitz domains for the
  symmetric stable processes.
\newblock {\em Probab. Math. Statist.}, 22(2):419--441, 2002.

\bibitem{JerisonKenig:1995}
D.~Jerison and C.~E. Kenig.
\newblock The inhomogeneous {D}irichlet problem in {L}ipschitz domains.
\newblock {\em J. Funct. Anal.}, 130(1):161--219, 1995.

\bibitem{mclean2000strongly}
W.~McLean.
\newblock {\em Strongly elliptic systems and boundary integral equations}.
\newblock Cambridge university press, 2000.

\bibitem{NguyenSickel:18}
V.~K. Nguyen and W.~Sickel.
\newblock On a problem of {J}aak {P}eetre concerning pointwise multipliers of
  {B}esov spaces.
\newblock {\em Studia Math.}, 243(2):207--231, 2018.

\bibitem{Nikolskii}
S.~Nikol'skii.
\newblock {\em Approximation of functions of several variables and imbedding
  theorems}, volume 205.
\newblock Springer Science \& Business Media, 1975.

\bibitem{Nirenberg}
L.~Nirenberg.
\newblock Remarks on strongly elliptic partial differential equations.
\newblock {\em Comm. Pure Appl. Math.}, 8(4):648--674, 1955.

\bibitem{ROSe14}
X.~Ros-Oton and J.~Serra.
\newblock The {D}irichlet problem for the fractional {L}aplacian: regularity up
  to the boundary.
\newblock {\em J. Math. Pures Appl.}, 101(3):275 -- 302, 2014.

\bibitem{Savare98}
G.~Savar{\'e}.
\newblock {Regularity results for elliptic equations in Lipschitz domains}.
\newblock {\em J. Funct. Anal.}, 152(1):176--201, 1998.

\bibitem{Tartar07}
L.~Tartar.
\newblock {\em An introduction to Sobolev spaces and interpolation spaces},
  volume~3.
\newblock Springer Science \& Business Media, 2007.

\bibitem{Terracini:18}
S.~Terracini, G.~Tortone, and S.~Vita.
\newblock On {$s$}-harmonic functions on cones.
\newblock {\em Anal. PDE}, 11(7):1653--1691, 2018.

\bibitem{Triebel10}
H.~Triebel.
\newblock {\em Theory of Function Spaces}.
\newblock Modern Birkh{\"a}user Classics. Springer Basel, 2010.

\bibitem{VishikEskin}
M.~I. Vi{\v{s}}ik and G.~I. {\`E}skin.
\newblock Equations in convolutions in a bounded region.
\newblock {\em Uspehi Mat. Nauk}, 20(3 (123)):89--152, 1965.
\newblock English translation in {\em Russian Math. Surveys}, 20:86-151, 1965.

\end{thebibliography}
\end{document}